\documentclass[14pt]{article}
\usepackage[utf8]{inputenc}
\usepackage{pdfpages}
\usepackage{mathpazo}
\usepackage{amssymb, amsthm,amsmath}
\usepackage{url}

\usepackage[toc,page]{appendix}

\usepackage[Conny]{fncychap}
\usepackage{graphicx}
\usepackage{adjustbox}
\usepackage{comment}
\usepackage{etex}
\usepackage{mathrsfs}
\usepackage[english]{babel}

\usepackage[scaled=0.85]{beramono}
\usepackage{tikz-cd}
\usepackage{float}
\usepackage{fancyhdr}

\usepackage[left=0cm,right=0cm,top=2cm,bottom=2cm,margin=3cm]{geometry}
\usepackage{amsthm}
\usepackage{bigints}
\usepackage{tikz}
\usepackage[linktoc=all]{hyperref}
\usepackage[capitalise]{cleveref}
\hypersetup{colorlinks=true,linkcolor=blue,citecolor=red}
\usetikzlibrary{matrix,arrows,decorations.pathmorphing}
\usetikzlibrary{patterns,calc,angles,quotes}
\usepackage{tikz-cd}
\usepackage[utf8]{inputenc}
\tikzset{commutative diagrams/.cd}
\newtheorem{theorem}{Theorem}[subsection]

\newtheorem{lemma}[theorem]{Lemma}

\newtheorem{proposition}[theorem]{Proposition}

\theoremstyle{definition}
\newtheorem{definition}[theorem]{Definition}

\newtheorem{example}[theorem]{Example}

\newtheorem{remark}[theorem]{Remark}

\newtheorem{notation}[theorem]{Notation}

\AddToHook{env/definition/begin}{\crefalias{theorem}{definition}}
\AddToHook{env/proposition/begin}{\crefalias{theorem}{proposition}}
\AddToHook{env/lemma/begin}{\crefalias{theorem}{lemma}}
\AddToHook{env/remark/begin}{\crefalias{theorem}{remark}}
\AddToHook{env/claim/begin}{\crefalias{theorem}{claim}}
\AddToHook{env/corollary/begin}{\crefalias{theorem}{corollary}}
\AddToHook{env/example/begin}{\crefalias{theorem}{example}}
\newcommand{\CrrCopEal}{\op{Corr}(\Ca)^{\otimes}_{\E,\op{all}}}
\newcommand{\CrrCpEal}{\op{Corr}(\Ca)_{\E,\op{all}}}

\newcommand{\Ccc}{(\Ca^{\op{op}})^{\coprod,\op{op}}}

\newcommand{\E}{\mathcal{E}}
\newcommand{\D}{\mathcal{D}}

\newcommand{\K}{\mathcal{K}}

\renewcommand{\P}{\mathcal{P}}
\newcommand{\N}{\mathcal{N}}

\newcommand{\op}[1]{\operatorname{#1}}
\renewcommand{\S}{\mathcal{S}}

\newcommand{\C}{\mathcal{C}}

\newcommand{\U}{\mathcal{U}}

\newcommand{\Schf}{\op{Sch}_{fd}}

\newcommand{\F}{\mathcal{F}}

\newcommand{\dd}{\delta}
\newcommand{\I}{\mathscr{I}}

\newcommand{\Ca}{\mathcal{C}}

\newcommand{\X}{\mathcal{X}}

\newcommand{\bb}{\bullet}

\numberwithin{subsection}{section}

\newcommand{\sset}{\op{Set}_{\Delta}}

\renewcommand{\U}{\mathcal{U}}
\newcommand{\bx}{\square}
\newcommand{\Crt}{\op{Cart}}
\newcommand{\Cart}{\op{\C art}}
\newcommand{\Krt}{\op{\K art}}

\newcommand{\SH}{\mathcal{SH}}
\newcommand{\Y}{\mathcal{Y}}
\fancypagestyle{main}{\fancyhf{}
	\fancyhead[LE,RO]{\scriptsize Six-func on ind-pro cat}
	\fancyhead[RE,LO]{\scriptsize \rightmark }
	\fancyfoot[CE,CO]{\thepage}}

\title{Abstract Six-Functor Formalisms: Extension to Ind- and Pro- Categories and Functorial Cohomological Purity.}
\author{Chirantan Chowdhury}
\begin{document}

\maketitle{}
\begin{abstract}
In this article, we study two consequences of abstract six-functor formalisms. Firstly, we  show that an abstract six-functor formalism can be extended to specific Ind- and Pro- categories of geometric setups. As an application, we can define the motivic stable homotopy theory $\SH(-)$ for ind-pro-algebraic stacks such as the Hecke stack. Secondly, we also show that Cohomological Purity is functorial using the multisimplicial language developed by Liu-Zheng.
    
\end{abstract}
\tableofcontents
\section{Introduction}

Grothendieck introduced six-functor formalisms to understand duality statements from a categorical perspective for cohomology theories.
In recent years, the work of Liu-Zheng (\cite{liu2017enhanced}) and Mann (\cite{mann2022padic}) has contributed to the development of abstract six-functor formalisms using the language of higher category theory. The notion of abstract six-functor formalisms has seen its applications in various topics like motivic homotopy theory (\cite{Khan-Ravi_Generalised_Coh_Stacks},\cite{Chowdhury},\cite{chowdhury2025nonrepresentablesixfunctorformalisms}), arithmetic geometry (\cite{mann2022padic},\cite{berkovichmotives}), and D-modules (\cite{soor2025sixfunctorformalismquasicoherentsheaves}), to name a few. In the context of the  projective limit of algebraic stacks, Yaylali (\cite{yaylali2025rationalmotivesproalgebraicstacks}) has extended the six-functor formalism for rational motives. Motivated by the ideas in Yaylali's work, we extend the abstract six-functor formalism to Ind- and Pro-categories of geometric objects. As an application, we extend the definition of motivic homotopy theory to ind-pro algebraic stacks such as the Hecke Stack. In addition, we show that the usual cohomological purity is functorial using the multisimplicial techniques developed by Liu-Zheng (\cite{liu2017enhanced}).\\

Let us explain the key results of the paper in the following two subsections.

\subsection{Extension of 6FF to Ind- and Pro-Categories.}

Let $(\Ca,\E)$ be a geometric set-up. We deal with the setting of \textit{presentable six-functor formalisms}. A \textit{presentable six-functor formalism} is a  lax-symmetric monoidal functor 
\begin{equation}
    \D : \op{Corr}(\Ca)_{\E,\op{all}} \to \op{Pr}^L_{\op{cl}}
\end{equation} 
 where $\op{Pr}^L_{\op{cl}}$ is the $(\infty,1)$-category of presentable $(\infty,1)$-categories which are Cartesian closed and morphisms admitting left adjoints. 
In order to construct six-functor formalism for ind and pro systems, we only consider presentable six-functor formalism which arise from Nagata set-up (\cite{dauser2025uniquenesssixfunctorformalisms}[Definition 2.1]). In words, a Nagata set-up is a tuple $(\Ca,\E,\I,\P)$ where $\I$ and $\P$ are analogous to open immersion and proper maps, and $\E$ are morphisms which are separated and of finite type. In addition to this, we fix a class of morphisms $\S$ in $\Ca$ for which $f^*$ has a left adjoint $f_{\#}$ which satisfies base change and projection formula. The class $\S$ plays the analogue of smooth morphisms.\\

In \cite{yaylali2025rationalmotivesproalgebraicstacks}, Yaylali constructs rational motives $\op{DM}(-,\mathbf{Q})$ for  pro-algebraic stacks. In his work, for a pro-algebraic stack $\X:=\{ \X_i\}_{i \in I}$ with smooth transition maps (i.e. tame in \cite{yaylali2025rationalmotivesproalgebraicstacks}), he defines
\begin{equation}
    \op{DM}(\X,\mathbf{Q}):= \op{colim}^* \op{DM}(\X_i,\mathbf{Q})
\end{equation}
in $\op{Pr}^L_{\op{cl}}$. He then proceeds to construct exceptional pushforwards and base change. It turns out that one can only construct exceptional pushforwards for maps of inverse systems which are $!$-adjointable. In this case, the base change essentially holds for only projective systems which are indexed by $\mathbf{N}$. An example of a pro-algebraic stack is the classifying stack of the positive loop group $BL^+G$. The transition maps in this case are smooth surjections which fits into this setup.\\

In the case of ind systems, one usually does not consider the pullback extension but the exceptional pullback extension (\cite[Section 2.2]{Richarz_2021}). In other words, for an inductive system of algebraic stacks $\Y:= \{Y_j\}_{j \in J}$, one expects 
\begin{equation}
    \D(\Y):= \op{lim}_{\op{Cat}_{\infty}}^! \D(\Y_j) \cong \op{lim}^!_{\op{Pr}^R}\D(\Y_j)   \cong \op{colim}_{!,\op{Pr}^L} \D(\Y_j).
\end{equation}
where $\D$ is an abstract six-functor formalism  e.g. \'etale sheaves, $\SH(-), \op{DM}(-)$. In this case, one can only get pullback maps for such classes of maps that commute with exceptional pushforward level-wise.  An example of an ind-algebraic stack or an ind-scheme we keep in mind is the affine Grassmanian $\op{Gr}_G$, where the transition maps are closed immersions.  \\

Inspired by the explanation above and the examples, we define the following two categories, along with classes of morphisms, which are the key objects of this paper. 
\begin{definition}[\cref{definition:IndProcatdefinition,defintiion:IndProAdjedgesdefinition}]
    Let $K$ be a filtered $\infty$-category. According to the setup above, we define the following categories:
    \begin{enumerate}
        \item 
        \begin{enumerate}
        \item Let $\op{Pro}^K_{\S}(\Ca)$ be the full subcategory of $\op{Fun}(K,\Ca^{\op{op}})$ spanned by functors $f: K \to \Ca^{\op{op}}$ such that $f(e)\in \S$ for all edges $e \in K$. 
        \item Let $\op{Adj}_!^!(\E)$ be the set of edges $f \in \Delta^1 \times K \to \Ca^{\op{op}}$ in $\op{Pro}^K_{\S}(\Ca)$   with the following conditions: 
    \begin{enumerate}
    \item they are stable under pullbacks in $\op{Pro}^K_{\S}(\Ca)$.
    \item $f|_k : \Delta^1 \to \Ca^{\op{op}}$ belongs to $\E$ for all $k \in K$.
    \item For every $\tau: k \to k'$, the square
    \begin{equation}
    \begin{tikzcd}
        \D(f(0,k')) \arrow[r," p_!"] \arrow[d,"f_{k!}"] & \D(f(0,k)) \arrow[d,"f_{k'!}"]\\
        \D(f(1,k')) \arrow[r,"q_!"] & \D(f(1,k))
        \end{tikzcd}
    \end{equation}
    is horizontally right adjointable.
    \end{enumerate}
    \end{enumerate}
         \item
         \begin{enumerate}
     \item Let $\tilde{\P} \subset \P$ be a subclass of morphisms in $\E$ stable under pullbacks, compositions and contains isomorphisms. Let $\op{Ind}^K_{\tilde{\P}}(\Ca)$ be the full subcategory of $\op{Fun}(K,\Ca)$ spanned by functors $f: K^ \to \Ca$ such that $f(e) \in \E$ for all edges $e \in K$. 
     \item Let $\op{Adj}_!^*(\op{All})$ be the set of edges $f \in \Delta^1 \times K \to \Ca$ in $\op{Ind}^K_{\tilde{P}}(\Ca)$  with the following conditions:

    \begin{enumerate}
    \item they are stable under pullbacks in $\op{Ind}^K_{\tilde{\P}}(\Ca)$.
    \item For every $\tau: k' \to k$, the square  \ given by
    \begin{equation}
    \begin{tikzcd}
        \D(f(1,k)) \arrow[r," p_! \cong p_*"] \arrow[d,"f_{k}^*"] & \D(f(1,k')) \arrow[d,"f_{k'}^*"]\\
        \D(f(0,k)) \arrow[r,"q_! \cong q_*"] & \D(f(0,k'))
        \end{tikzcd}
    \end{equation}
    commutes. 
    \end{enumerate}
    \end{enumerate}
    \end{enumerate}
    
\end{definition}
\begin{remark}
    One sees that objects of $\op{Pro}^K_{\S}(\Ca)$ are a generalized notion of tame pro $K$-algebraic stack defined  in  \cite{yaylali2025rationalmotivesproalgebraicstacks}. In general, the above defintions are motivated from the example of Ind- and Pro- Objects in the setting of Motivic Langlands as studied in \cite{richarzscholbachintersectionmotive} and \cite{zhu2025tamecategoricallocallanglands}.
\end{remark}

The main result of this section is as follows:

\begin{theorem}[\cref{theorem:IndProextensions}]\label{theorem:IntroIndProextension}
Let 
\begin{equation}
    \D : \op{Corr}(\Ca)_{\E,\op{all}} \to \op{Pr}^L_{\op{cl}}
\end{equation}
be a presentable six-functor formalism arising from Nagata set-up. 

\begin{enumerate}
    \item Assume $\S$ satisfies cohomological purity, i.e. the canonical transformation $\op{Pur}_f : f_{\#} \to f_!\Sigma_f$ induces an equivalence $f_{\#} \cong f_!$. Then $\D$ can be extended to a presentable six-functor formalism:
    \begin{equation}
        \D_{\op{pro}} : \op{Corr}(\op{Pro}^K_{\S}(\Ca))_{\op{Adj}^!_!(\op{\E}),\op{All}} \to \op{Pr}^L_{\op{cl}}
    \end{equation}
    \item $\D$ can be extended to a presentable six-functor formalism
    \begin{equation}
        \D_{\op{ind}}: \op{Corr}(\op{Ind}^K_{\tilde{\P}}(\Ca))_{\E,\op{Adj}^*_!(\op{All})} \to \op{Pr}^L_{\op{cl}}
    \end{equation}
\end{enumerate}
\end{theorem}

\begin{remark}
\begin{enumerate}
 \item    The proof of the theorem involves using the multisimplicial set formalism developed by Liu-Zheng (\cite{liu2017enhanced}). 
\item For an object $\X:=\{\X_k\}_{k \in K}$ in $\op{Pro}^K_{\S}(\Ca)$, we have 
\begin{equation}
    \D_{\op{pro}}(\X) \cong \op{colim}_{\op{Pr}^L_{\op{cl}}}^* \D(\X_k)
\end{equation}
Similarly, for an object $\Y:= \{ \Y_k\}_{k \in K}$ in $\op{Ind}^K_{\tilde{\P}}(\Ca)$, we have
\begin{equation}
    \D_{\op{ind}}(\Y) \cong \op{lim}_{\op{Pr}^L_{\op{cl}}}^! \D(\Y_k) \cong \op{lim}_{\op{Cat}_{\infty}}^! \D(\Y_k)
\end{equation}
Thus, the pro-extensions  and ind-extensions align with what we discussed earlier.
\item The above extension procedure can also be done to other geometric setups which arise as extension from geometric setups coming from Nagata Setup. Such extensions are addressed in \cref{theorem:proextensiongeoandeceptionalpairs} and \cref{Theorem:ExtendingIndGeneralSetup}.

\item  The $\D_{\op{pro}}$ functor extend the results of \cite{yaylali2025rationalmotivesproalgebraicstacks} to arbitrary simplicial set $K$ rather than $\mathbf{N}$.  The advantage of using the the multisimplicial sets allows us to circumvent using arbitrary simplicial sets. In particular, Yaylali crucially uses the indexing set $\mathbf{N}$ to show that
\begin{equation}
    \op{colim}^!\D(\X_k) \cong \op{colim}^*\D(\X_k) .
\end{equation}
Using the functorial purity that we will discuss in the next subsection, we extend this equivalence to an arbitrary indexing set $K$ (\cref{uppershriekupperstarsmoothequivalence}).

    \end{enumerate}
\end{remark}
As an application, we can extend the six-functor formalism for motivic stable homotopy theory $\SH(-)$ to ind-pro algebraic stacks. In this article, we consider extending the Lisse-extend cohomology theory $\SH(-)$. The six-functor formalism of algebraic stacks for such theories has been developed by Chowdhury (\cite{Chowdhury}),Chowdhury-D'Angelo (\cite{chowdhury2025nonrepresentablesixfunctorformalisms}), and Khan-Ravi (\cite{Khan-Ravi_Generalised_Coh_Stacks}).\\

\begin{notation}Let 
\begin{enumerate}
\item $K$ be a filtered $\infty$-category admitting an initial object.
\item $\op{IndPro}^K_{\op{cl,sm}}(\op{AlgSt}):= \op{Ind}^K_{\op{Adj}^!_!(\op{Pro}(\op{cl}))}(\op{Ind}^K_{\op{sm}}(\Ca))$ be the category of inductive system of projective system of algebraic stacks with smooth and closed transition maps. 
\item Let $\op{IndPro}^!_!(\E):= \op{Adj}^!_!(\E)$ and $\op{Ind}^*_!(\op{All}):= \op{Adj}^*_!(\op{All})$.
\end{enumerate}
\end{notation}

Applying \cref{theorem:IntroIndProextension}, to the abstract six-functor formalisms $\SH(-)$, we get the following 
\begin{theorem}[\cref{theorem:sixfuncindproSH}]
The abstract six-functor formalism $\SH(-)$ extend to 

\begin{equation}
  \SH(-) : \op{Corr}(\op{IndPro}^K_{\op{cl,sm}}(\op{AlgSt}))_{\op{IndPro}^!_!(\E),\op{Ind}^*_!(\op{All})} \to \op{Pr}^L_{\op{cl}}
\end{equation}
\end{theorem}

\begin{remark}
    The above extension procedure also applies to the six-functor formalism $\op{DM}(-)$, which assembles Yaylali's result together and also generalizes to an arbitrary indexing set.
\end{remark}

\subsection{Functorial Cohomological Purity.}
In the setting of abstract six-functor formalism, for a morphism $f: X \to Y$ in $\E \cap \S$, we have the transformation 
\begin{equation}
    \op{Pur}_f : f_{\#} \to f_! \circ \Sigma_f \,\,;\,\, \Sigma_f:= p_{1\#}\delta_{f!}
\end{equation}
which arises from the exchange transformation applied to the diagram
\begin{equation}
    \begin{tikzcd}
        X \arrow[ddr,bend right=50,equal] \arrow[drr,bend left =50, equal] \arrow[dr,"\delta_f"] & {} & {} \\
        {} & X \times_Y X \arrow[r,"p_1"] \arrow[d,"p_2"] & X \arrow[d,"f"] \\
        {} & X \arrow[r,"f"] & Y
    \end{tikzcd}
\end{equation}

We say that $\D$ satisfies \textit{cohomological purity} if for every $f \in \E \cap \S$, $\op{Pur}_f$ and $\Sigma_f$ are equivalences. Thus we have an equivalence $f_{\#} \cong f_!$. The goal of this section is to upgrade this equivalence to an equivalence of functors. The main statement of this section is the following:
\begin{theorem}[\cref{Theorem: FunctorialPurityFullForm}]\label{Theorem: FunctorialPurityFullFormIntro}
  Let $(\Ca,\E,\I,\P)$ be a Nagata setup and let 
    \begin{equation}
        \D:  \op{Corr}(\Ca)^{\otimes}_{\E,\op{all}} \to \op{Cat}_{\infty}
         \end{equation}
        be an abstract six-functor formalism  arising from a Nagata set-up with the following properties :
        \begin{enumerate}
            \item There exists a set of edges $\mathcal{S}$ stable under pullbacks and compositions containing $\I$ such that $f^*$ for $f \in \S$ has a left adjoint $f_{\#}$ which satisfies base change and projection formula.
            \item (\textit{Ambidexterity:}) The natural transformation $\op{Ex}_{\#!}$ is an equivalence for squares :
            \begin{equation}
                \begin{tikzcd}
                    X' \arrow[r,"f'" ] \arrow[d,"g'"] & X \arrow[d,"g"]\\
                    X \arrow[r,"f"] & Y
                \end{tikzcd}
            \end{equation}
            where $f,f' \in \S, g,g' \in \E$.
         \end{enumerate}
        Then the cohomological purity equivalence $f_{\#} \to f_!\Sigma_f$ can be upgraded to a natural transformation :
        \begin{equation}
            \op{Pur}_{\D} : \Delta^1 \times \Ca_{\E \cap \S} \to \op{Cat}_{\infty}
        \end{equation}
        which is a natural transformation between $\D_{\#}$ and $\D^{\Sigma}_!$ where $\D^{\Sigma}_!$ on the level of morphisms sends $f \mapsto f_!\Sigma_f$.
\end{theorem}

The main content of this proof is to functorially construct a functor $\D^{\Sigma}_! : \Ca_{\E \cap \S} \to \op{Cat}_{\infty}$ which sends $X \mapsto \D(X)$ and $f \mapsto f_!\Sigma_f$. This relies on  proving a variant of the compactification theorem proved by Liu-Zheng (\cite[Theorem 1.5.1]{liu2017enhanced} which is also reproved in \cite[Thm B]{chowdhury2025sixfunctorformalismsii}) in the setting of multi-simplicial sets. 

\begin{remark}
    It is, however, expected to construct the natural transformation in general without $\op{Ex}_{\#!}$ to be an equivalence. This only shall make sense when the target is $\op{Cat}_{(\infty,2)}$. (Refer to \cref{remark:extendingpuritytocat2} for more details).
\end{remark}

\subsection{Content of the Paper.}
Let us describe the contents of this paper:
\begin{enumerate}
    \item In Section 2, we recall the important notions of  correspondences and presentable six-functor formalisms. We also briefly describe how the convolution product is encoded in the setting of the abstract six-functor formalism.
    \item In Section 3, we describe the  construction of functorial cohomological purity. The major part of this section is to construct $\D^{\Sigma}_!$, which leads us to prove the variant of the $\infty$-compactification theorem. The construction of $\op{Pur}_{\D}$ follows easily from that. 
    \item In Section 4, we extend our abstract six-functor formalisms to inductive and projective systems of our objects in a geometric setup. We introduce the setup, definitions and prove some technical propositions which help us to prove the extension theorems (\cref{theorem:IndProextensions}).
    \item In the last section, we apply the extension results in the setting of  motivic homotopy theory.
\end{enumerate}

\subsection*{Acknowledgements}

The paper was written while the author was a Postdoctoral researcher under Prof. Dr. Timo Richarz at the University of TU Darmstadt. C.Chowdhury acknowledges support (through Timo Richarz) by the European Research Council (ERC) under Horizon Europe (grant agreement nº 101040935), by the Deutsche Forschungsgemeinschaft (DFG, German Research Foundation) TRR 326 \textit{Geometry and Arithmetic of Uniformized Structures}, project number 444845124, and the LOEWE professorship in Algebra, project number LOEWE/4b//519/05/01.002(0004)/87. \\

The author would like to thank Can Yaylali for discussions regarding his work on rational motives for pro-algebraic stacks. The author would  also like to thank Alessandro D'Angelo, R{\i}zacan \c{C}ilo\u{g}lu, and Timo Richarz for helpful discussions regarding the paper. 

\subsection*{Conventions}

The paper heavily relies on the development on Abstract six-functor formalisms developed by Liu-Zheng and Mann. We shall be  referencing the language of multi-simplicial sets developed by Liu-Zheng in \cite[Section 1.3]{liu2017enhanced} and which is revised in \cite[Section 2]{chowdhury2025sixfunctorformalismsii}. We shall freely use the language of $\infty$-categories developed by Lurie in (\cite{HTT},\cite{HA},\cite{kerodon} and \cite{SAG}).

\section{Correspondences and Presentable Six-Functor Formalisms.}

In this section, we recall the definitions and relevant statements concerning the $(\infty,1)$-category of correspondences and the notion of abstract six-functor formalisms formalized by Liu-Zheng and Mann. We impose the presentability condition on the formalism by imposing that it is valued in $\op{Pr}^L_{\op{cl}}$. Lastly, we recall the Nagata set-up and the construction theorem of such a formalism. This plays a crucial role in proving the extension theorems in Section 4.  
    \subsection{ The $(\infty,1)$-category of correspondences..}
    \begin{definition}\cite[Definition A.5.1]{mann2022padic}
        A \textit{geometric setup} is a pair $(\Ca,\E)$ where $\Ca$ is an $\infty$-category and $E$ is a homotopy class of edges in $\Ca$ satisfying
        \begin{enumerate}
            \item $\E$ contains all isomorphisms and is stable under compositions,

           \item Pullbacks of $\E$ exist and remain in $\E$.
        \end{enumerate}
        We call such a class $\E$ to be \textit{weakly stable}. In addition to the above conditions, if $\E$ satisfies the following property :  for two composable morphisms $\alpha_2,\alpha_1 \in \Ca$ with $\alpha_1 \in E$ then, $\alpha_2 \in \E \leftrightarrow \alpha_2 \circ \alpha_1 \in \E$, we call $\E$ to be \textit{admissible}. 
    \end{definition}
    \begin{definition}\cite[Definition A.5.2]{mann2022padic}
        \begin{enumerate}
            \item Let $C(\Delta^n) \subset \Delta^n \times (\Delta^n)^{op}$ be the full-subcategory spanned by $([i],[j])$ with $i \le j$.
            \item An edge in $C(\Delta^n)$ is \textit{vertical} (resp. horizontal) if the projection to the second (first) factor is degenerate.
            \item A square in $C(\Delta^n)$ is \textit{exact} if it is both a pullback and a pushout square.
            \item For any simplicial set $K$, we define 
            \begin{equation*}
                C(K): = \op{colim}_{(n,\sigma) \in \Delta_{/K}}C(\Delta^n).
            \end{equation*}
            Thus $C$ can be visualized as an endofunctor on the category of simplicial sets. It admits a right adjoint 
            \begin{equation*}
                B: \sset \to \sset
            \end{equation*}
            defined by 
            \begin{equation*}
                K \mapsto B(K) =\{B(K)_n:= \op{Hom}_{\sset}(C(\Delta^n),K) \}.
            \end{equation*}
            \item Given a geometric setup $(\Ca,\E)$, define 
            \begin{equation*}
                \op{Corr}(\Ca)_{E,\op{all}} \subset B(\Ca)
            \end{equation*}
            to be the sub-simplicial set whose $n$ simplices are maps $C(\Delta^n) \to \Ca$ which sends 
            \begin{enumerate}
                \item vertical edges to $E$,
                \item exact squares to pullback squares.
            \end{enumerate}
        \end{enumerate}
    \end{definition}
 
    \begin{remark}
Let $(\Ca,\E)$ be a geometric setup. Then the lower simplices of $\op{Corr}(\Ca)_{\E,\op{all}}$ look like as follows:
     \begin{itemize}
         \item The $0$-simplices are objects of $\Ca$.
         \item A $1$-simplex i.e an edge between $X_0$ and $X_1$ where $X_0,X_1 \in \Ca$ is a diagram of the form :
         \begin{equation*}
             \begin{tikzcd}
                 X_0 & X_{01}\arrow[l] \arrow[d,"f"] \\
                 {} & X_1
             \end{tikzcd}
         \end{equation*}
         where $f \in E$.
         \item A $2$-simplex in $\op{Corr}(\Ca)_{\E,\op{all}}$ looks like as follows:
         \begin{equation*}
             \begin{tikzcd}
                 X_0 & X_{01} \arrow[l] \arrow[d,"f"]  \arrow[dr, phantom, "\square"]  & X_{02} \arrow[l] \arrow[d,"g"] \\
                 {} & X_{11}  & X_{12} \arrow[d,"h"] \arrow[l] \\
                 {} & {} & X_{22}
             \end{tikzcd}
         \end{equation*}
         where $f,g,h\in E$ and $\square$ is a pullback square.
     \end{itemize}
         The $\infty$-category $\op{Corr}(\Ca)_{\E,\op{all}}$ admits the following canonical maps :
    \begin{equation*}
        \Ca_{\E} \to \op{Corr}(\Ca)_{\E,\op{all}} \quad,\quad \Ca^{op} \to \op{Corr}(\Ca)_{\E,\op{all}}
    \end{equation*}

         \end{remark}

       We list some well known facts about the $\op{Corr}(\Ca)_{\E,\op{all}}$ 

       \begin{proposition}\cite[Section 4]{chowdhury2025sixfunctorformalismsiiiconstruction}
       \begin{enumerate}
           \item  The simplicial set $\op{Corr}(\Ca)_{\E,\op{all}}$ is an $(\infty,1)$-category. 
           \item  Suppose $\Ca$ admits finite coproducts. Then the map $\Ca^{\op{op}} \to \op{Corr}(\Ca)_{\E,\op{all}}$ preserves finite products.
           \end{enumerate}
       \end{proposition}
To define abstract $6$-functor formalisms, one needs to encode a symmetric monoidal structure on $\op{Corr}(\Ca)_{\E,\op{all}}$.
\begin{notation}
    let $(\Ca,\E)$ be a  geometric setup. We define a geometric setup on the $\infty$-category $(\Ca^{\op{op}})^{\coprod,\op{op}}$. We write an edge $f$ of $\Ccc$ in the form $\{Y_j\}_{1\le j \le n} \to \{X_i\}_{1 \le i \le m}$ lying over $\alpha : [m] \to [n]$. We define two sets of $\E^{+},\E^{-}$ as follows:
    \begin{itemize}
        \item $\E^+$ consists of $f$ such that the induced age $Y_{\alpha(i)} \to X_i$ belongs to $\E$ for every $i \in \alpha^{-1}(\langle n \rangle ^0)$,
        \item $\E^-$ is subset of $\E^+$ where $\alpha$ is degenerate.
    \end{itemize}
    We shall denote :
    \begin{equation*}
        \op{Corr}(\Ca)^{\otimes}_{\E,\op{all}} := \op{Corr}(\Ccc)_{\E^{-},\op{all}^+}
    \end{equation*}
\end{notation}
The following proposition ensures that the above notation does encode a symmetric monoidal structure.
\begin{proposition}
    Let $(\Ca,\E)$ be an $\infty$-category such that $\Ca$ admits finite products. Then 
    \begin{equation*}
        p : \CrrCopEal \to N(\op{Fin}_*)
    \end{equation*}
    is a coCartesian symmetric monoidal $\infty$-category whose underlying $\infty$-category is $\op{Corr}(\Ca)_{\E,\op{all}}$.
\end{proposition}
\subsection{Presentable six-functor formalisms.}
       \begin{definition}
    Let $(\Ca,\E)$ be a geometric setup where $\Ca$ admits finite products. Then a \textit{presentable 6-functor formalism} is a morphism of $\infty$-operads :
    \begin{equation*}
        \D : \CrrCopEal \to \op{Pr}^{L\otimes}_{\op{cl}}.
    \end{equation*}
    Given a presentable 6-functor formalism, we introduce the following notations :
    \begin{enumerate}
        \item Restricting to the sub-operad $\Ccc$, we get a functor :
        \begin{equation*}
            \D^{*\otimes} : \Ca^{\op{op},\coprod} \to \op{Pr}^{L\otimes}
        \end{equation*}
        This is equivalent to the functor 
        \begin{equation*}
            \D^{*\otimes} : \Ca^{\op{op}} \to \op{CAlg}(\op{Pr}^L).
        \end{equation*} Taking right adjoints, we also have the pushforward functor 
        \begin{equation*}
            \D_* : \Ca \to \op{Pr}^R.
        \end{equation*}
        \item As $\D(X):=\D_{(\Ca,\E)}(X)$ is symmetric monoidal for every $X \in \Ca$, we get a tensor product structure :
        \begin{equation*}
            - \otimes - : \D(X) \times \D(X) \to \D(X).
        \end{equation*} As $\D(X)$ is closed, we also have the Internal Hom functor 
        \begin{equation*}
            \op{Hom}(-,-) : \D(X) \times \D(X) \to \D(X).
        \end{equation*}
        \item Using the inclusion $\Ca_{\E} \to \CrrCpEal$, we get the following functor :
        \begin{equation*}
           \D_!: \Ca_{\E} \to \op{Pr}^L.
        \end{equation*}
        Taking right adjoints gives us the functor:
        \begin{equation}
            \D^! : \Ca_{\E} \to \op{Pr}^R.
        \end{equation}
    \end{enumerate}
\end{definition}

\begin{remark}
    Here $\op{Pr}^L_{\op{cl}}$ is the $(\infty,1)$-category of presentable $(\infty,1)$-categories which are closed and morphism are left adjoints. This admits a monoidal structure via the Lurie Tensor product. 
\end{remark}
Let us briefly describe how the six functor formalism encodes properties such as the projection formula and base change.
\begin{enumerate}
    \item \textbf{Projection formula:} Let $f: X \to Y$ be a morphism in $E$, we consider the diagram :
    \begin{equation*}
        \begin{tikzcd}
            (X, Y) \arrow[r,"\sigma"] \arrow[d,"\tau"] & (Y,Y) \arrow[d,"\tau'"] \\
            X \arrow[r,"f"] & Y
        \end{tikzcd}
    \end{equation*}
    where :
    \begin{itemize}
        \item $\sigma':$\begin{equation*}
             \begin{tikzcd}
                (X,Y) & \arrow[l," \op{id}"] (X,Y) \arrow[d,"f_1"] \\
                {} & (Y,Y).
            \end{tikzcd}
        \end{equation*}
        where $f_1 = (f,\op{id})$.
        \item $\tau:$ \begin{equation*}
             \begin{tikzcd}
                (X,Y) & \arrow[l,"f_2"]  X \arrow[d,"\op{id}"] \\
               {} & X.
                \end{tikzcd}
        \end{equation*}
        where $f_2=(\op{id},f)$
        \item  $\tau':$\begin{equation*}
             \begin{tikzcd}
                (Y,Y) & Y \arrow[d,"\op{id}"] \arrow[l,"f_3"] \\
                {} & Y.
            \end{tikzcd}
        \end{equation*}
        where $f_3 = (\op{id},\op{id})$.
    \end{itemize}

This gives a morphism $\Delta^1 \times \Delta^1 \to \CrrCopEal$.  Applying $\D_{(\Ca,\E)}$, we get the following commutative square in $\op{Cat}_{\infty}$:
\begin{equation*}
    \begin{tikzcd}
        \D(X) \times \D(Y) \arrow[r,"f_! \times \op{id}"] \arrow[d, "\op{id} \otimes f^* "] & \D(Y) \times \D(Y) \arrow[d,"-\otimes-"]\\
        \D(X) \arrow[r,"f_!"] & \D(Y).
    \end{tikzcd}
\end{equation*}
which is the projection formula :
\begin{equation*}
    f_!((-) \otimes f^*(-)) \cong f_!(-) \otimes (-).
\end{equation*}
\item \textbf{Base change :} 
Let 
\begin{equation*}
    \begin{tikzcd}
        X' \arrow[d,"p'"] \arrow[r,"q'"] & X \arrow[d,"p"] \\
        Y' \arrow[r,"q"] & Y
    \end{tikzcd}
\end{equation*}
be a cartesian square with $p',p \in \E$. We have a commutative square in $\CrrCopEal$ 
\begin{equation*}
    \begin{tikzcd}
        X \arrow[r] \arrow[d] & X' \arrow[d] \\
        Y\arrow[r,] & Y'
    \end{tikzcd}
\end{equation*}
which is comprised of two $2$-simplices
\begin{itemize}
    \item $\sigma_1:$
    \begin{equation*}
        \begin{tikzcd}
            X & \arrow[d,"p"] \arrow[l,"\op{id}"] X & X' \arrow[l,"q'"] \arrow[d,"p'"] \\
            {} & Y & Y'\arrow[l,"q"] \arrow[d,"\op{id}"] \\
            {} & {} & Y'.
         \end{tikzcd}
    \end{equation*}
    \item $\sigma_2':$
    \begin{equation*}
        \begin{tikzcd}
            X & X' \arrow[l,"q'"] \arrow[d,"\op{id}"] & X' \arrow[d,"\op{id}"] \arrow[l,"\op{id}"] \\
            {} & X' & X'\arrow[l,"\op{id}"]\arrow[d,"p'"]\\
            {} & {} & Y'.
        \end{tikzcd}
    \end{equation*}
    Applying $\D_{(\Ca,\E)}$ to the square, we get the commutative square :
    \begin{equation*}
        \begin{tikzcd}
            \D(X) \arrow[r,"q^{'*}"] \arrow[d,"p_!"] & \D(X') \arrow[d,"p'_!"]\\
            \D(Y) \arrow[r,"q^*"] & \D(Y')
        \end{tikzcd}
    \end{equation*}
    in $\op{Cat}_{\infty}$. Spelling this out, we get the base change equivalence $Ex^*_!$:
    \begin{equation}\label{App.:_Base_Change_Ex*_!}
    	q^*p_!\stackrel{Ex^*_!}{\simeq} p_!'q^{*'}
    \end{equation}
    
\end{itemize}
\end{enumerate}

Let us now recall the general setup of how six-functor formalisms. We recall the notion of Nagata set-up from \cite{dauser2025uniquenesssixfunctorformalisms}.
 \begin{definition}\cite[Definition 2.1]{dauser2025uniquenesssixfunctorformalisms}
     A \textit{Nagata setup} $(\Ca,\E,\I,\P)$ is a geometric setup $(\Ca,\E)$ together with two subsets $\I,\P \subset \E$ such that :
     \begin{enumerate}
         \item $(\Ca,\I)$ and $(\Ca,\P)$ are geometric setups,
         \item Every morphism $f \in E$ admits a decomposition $f = \bar{f} \circ j$ where $j \in \I$ and $\bar{f} \in \P$.
         \item  Given $f : X \to Y$ in $\Ca$ and $g: Y \to Z$ in $\I$($\P$) then $f \in \I$($\P$) iff $g \circ f \in \I$($\P$).
         \item Every morphism $f \in \I \cap \P$ is $k$-truncated for some $k \ge -2$. 
         \end{enumerate}
 \end{definition}
 \begin{example}
     Let $\Ca=\op{Sch}$ be the category of Noetherian schemes of finite Krull dimension. Let $\E$ be morphisms that are separated of finite type. Choosing $\I$ as open immersions and $\P$ as proper morphisms, by Nagata compactification, we see that the tuple $(\Ca,\E,\I,\P)$ is a Nagata setup. 
 \end{example}
Now we proceed to state the main theorem, which enables us to construct $3$/$6$-functor formalisms. This theorem uses all the results from the previous articles.

\begin{theorem}\cite[Theorem 5.2.5]{chowdhury2025sixfunctorformalismsiiiconstruction}\label{theorem: mainconstructiontheorem}
Let $(\Ca,\E,\I,\P)$ be a Nagata setup. Let
\begin{equation}
   \D^{*\otimes}: \Ca^{op} \to \op{CAlg}(\op{Pr}^L_{\op{st},\op{cl}})
\end{equation}
be a functor satisfying the following conditions : 
		\begin{enumerate}
			\item  For any  morphism in $\I$, $f^*$ has a left adjoint $f_{\#}$ such that: 
			\begin{enumerate}
				\item  ( $\I$-projection formula) For  any $E \in \D(Y)$ and $B \in \D(X)$, the natural map formed by adjunction
				\begin{equation}\label{smothpro}
					f_{\#}(E \otimes f^*(B)) \to f_{\#}E \otimes B \end{equation} is an equivalence.
				\item  ($\I$-base change) For a cartesian square
				\begin{equation}\label{8.1}
					\begin{tikzcd}
						X' \arrow[r,"f'"] \arrow[d,"g'"] & Y' \arrow[d,"g"] \\
						X \arrow[r,"f"] & Y
					\end{tikzcd}
				\end{equation}
				with $f \in \I$, the commutative square  
				
				\begin{equation}\label{8.2}
					\begin{tikzcd}
						\D(X') & \D(Y') \arrow[l,"\lbrace f '\rbrace^*"]  \\
						\D(X ) \arrow[u,"\lbrace g'\rbrace^*"] & \D(Y) \arrow[l,"f^*"] \arrow[u,"g^*"]
					\end{tikzcd}
				\end{equation}
				is horizontally left-adjointable, i.e., there exists a commutative square
				\begin{equation}\label{8.3}
					\begin{tikzcd}
						\D(X') \arrow[r,"\lbrace f'\rbrace_{\#}"]& \D(Y') \\
						\D(X ) \arrow[u,"\lbrace g'\rbrace^*"] \arrow[r,"f_{\#}"] & \D(Y) \arrow[u,"g^*"]
					\end{tikzcd}
				\end{equation}
			\end{enumerate}
			\item For $f: Y \to X$ a  morphism  in $\P$ , $f^*$ admits a right adjoint functor $f_*$ with the following properties:
			\begin{enumerate}
				\item ($\P$-projection formula) For $E \in \D(Y)$ and $B \in \D(X)$, the natural map 
				\begin{equation}\label{proppro}
					f_*(E) \otimes B \to f_*(E \otimes f^*(B)) 
				\end{equation}
				is an equivalence.
				\item ($\P$ base change) For the cartesian square in \cref{8.1} with $f \in \P$, the induced pullback square \cref{8.2} is horizontally right adjointable. In other words, the square commutes
				\begin{equation}\label{8.4}
					\begin{tikzcd}
						\D(X') \arrow[r," f'_*"]& \D(Y') \\
						\D(X ) \arrow[u,"\lbrace g'\rbrace^*"] \arrow[r,"f_*"] & \D(Y) \arrow[u,"g^*"]
					\end{tikzcd}
				\end{equation}
			\end{enumerate} 
			\item (Support property) For a cartesian diagram in \cref{8.1} where $f \in \I$ and $g \in \P$,the commutative diagram in \cref{8.3} written as  square
			\begin{equation}\label{8.5}
				\begin{tikzcd}
					\D(X) \arrow[r,"f_{\#}"] \arrow[d,"\lbrace g' \rbrace ^*"] & \D(Y) \arrow[d,"g^*"] \\
					\D(X') \arrow[r,"\lbrace f' \rbrace_{\#}"] & \D(Y')
				\end{tikzcd}
			\end{equation} 
			is horizontally right adjointable, i.e., the square 
			
			\begin{equation}\label{8.6}
				\begin{tikzcd}
					\D(X') \arrow[r,"\lbrace f' \rbrace_{\#}"] \arrow[d,"g'_*"] & \D(Y') \arrow[d,"g_*"] \\
					\D(X) \arrow[r,"f_{\#}"] & \D(Y) 
				\end{tikzcd}
			\end{equation}
			commutes.
			
		\end{enumerate}
    Then $\D^{*\otimes}$ can be upgraded to a $3$-functor formalism : 
    \begin{equation}
        \D_{(\Ca,\E)} : \op{Corr}(\Ca)_{\E,\op{all}}^{\otimes} \to \op{Cat}_{\infty}^{\otimes}   \end{equation}
        such that for all $f \in \P$, $f_!=f_*$ and $f\in \I$, $f_!=f_{\#}$.
Moreover, if the $3$-functor formalism satisfies the additional assumptions :
\begin{enumerate}
    \item For every $X \in \Ca$, $\D(X)$ is closed.
    \item For every $f: X \to Y$ in $\Ca$, $f^*$ admits a right adjoint $f_*$.
    \item For every $f: X \to Y$ in $\P$, $f_*$ admits a right adjoint $f^!$.
    \end{enumerate}
Then the $3$-functor formalism is a $6$-functor formalism.         
	\end{theorem}

\begin{remark}
    In \cite{chowdhury2025sixfunctorformalismsiiiconstruction}, we reprove this construction of Liu-Zheng which uses the technical lemma of constructing functors (\cite[Theorem 1.0.4]{chowdhury2023sixfunctorformalismsi}), the infinity-categorical compactification (\cite[Theorem 1.0.3]{chowdhury2025sixfunctorformalismsii}) followed by Partial Adjoints (\cite[Theorem 3.2.1]{chowdhury2025sixfunctorformalismsiiiconstruction})
\end{remark}

\subsection{The convolution product.}

The $\infty$-category of correspondences admits commutative and associative algebra objects. In this subsection, we discuss some examples and, in particular, the convolution product that usually arises from a six-functor formalism. We restate the results as stated in \cite[Remark 8.12]{zhu2025tamecategoricallocallanglands}.

\begin{lemma}
    Let $(\Ca,\E)$ be a geometric setup and let $X \in \op{Alg}_{E_k}(\Ca)$. Then $X \in \op{Alg}_{E_k}(\CrrCpEal)$. In particular, commuative and associative algebra objects of $\Ca$ are also commutative and associative algebra objects of $\op{Corr}(\Ca)_{\E,\op{all}}$.
\end{lemma}

\begin{proof}

This follows from the fact that the morphism $\Ca^{\op{op}} \to \CrrCpEal$ is symmetric monoidal.
\end{proof}

 The next lemma is about understanding the convolution product in an abstract six-functor formalism. 

\begin{theorem}\label{abstractconvolutionproduct}
    Let $\E_1.\E_2$ be two sets of edges in $\Ca$ such that $(\Ca,\E_1)$ and $(\Ca,\E_2)$ are geometric setups.  Let $f:X \to Y$ in $\Ca$, such that $f,\Delta_f \in \E_1 \cap \E_2$,  then $X \times_Y X \in \op{Ass}(\op{Corr}(\Ca)_{\E_1,\E_2})$.   
\end{theorem}

The theorem shall follow from the following slight generalization of \cite[Proposition 2.3.9]{heyer20246functorformalismssmoothrepresentations}. For the sake of completeness, we state and prove the proposition that follows similar arguments to those stated in the above reference.

\begin{proposition}\label{dualizableobjinCorr}
    Let $X \in \op{Corr}(\Ca)_{\E_1,\E_2}$ such that $p_X: X \to *, \Delta_X \in \E_1 \cap \E_2$, then $X$ is dualizable in $\op{Corr}(\Ca)_{\E_1,\E_2}$ with $X^{\vee}=X$. 
\end{proposition}
\begin{proof}
Given the conditions of morphisms, we have the following morphisms in $\op{Corr}(\Ca)_{\E_1,\E_2}$ :
\begin{equation}
    \op{coev}_X :\begin{tikzcd}
        * & X \arrow[l,"p_X"]\arrow[d,"\Delta_f"] \\
        {} & X \times X
    \end{tikzcd}\quad\quad \op{ev}_X:
    \begin{tikzcd}
        X \times X & X \arrow[l,"\Delta_f"] \arrow[d,"p_X"] \\
        {} & *.
    \end{tikzcd}
\end{equation}
We claim these are coevaluation and evaluation maps for $X$. Thus, we need to show the identities in the condition of dualizability. We shall show one of the identities, i.e., the composition:
\begin{equation}
    X = X \times \op{id} \xrightarrow{ \op{coev}_X \times \op{id}} X \times X \times X \xrightarrow{\op{id}_X \times \op{ev}_X} X
\end{equation}
is equivalent to $\op{id}_X$.  Writing the above composition explicitly as a diagram of correspondences yields the following diagram:
\begin{equation}
    \begin{tikzcd}
        X & X \times X \arrow[l,"p_X \times \op{id}_X"] \arrow[d,"\Delta_X \times \op{id}_X"] & X \arrow[l,"\Delta_X"] \arrow[d,"\Delta_X"] \\
        {} & X \times X \times X  & X  \times X \arrow[l,"\op{id}_X \times \Delta_X"] \arrow[d,"\op{id}_X \times p_X"] \\
        {} & {} & X
    \end{tikzcd}
\end{equation}
 The above composition is equivalent to $\op{id}_X$ if the above diagram is a $2$-simplex in  $\op{Corr}(\Ca)_{\E_1,\E_2}$. We need to show that the square in the above diagram is a pullback square. 

\end{proof}

\begin{proof}[Proof of \cref{abstractconvolutionproduct}] Let $f: X \to Y$. Then $f$ induces a symmetric monoidal functor (follows from description of coCartesian morphisms \cite[Lemma 4.1.11]{chowdhury2025sixfunctorformalismsiiiconstruction})
\begin{equation}
    \op{Corr}(\Ca_{/Y})_{\E_1,\E_2} \to \op{Corr}(\Ca)_{\E_1,\E_2}.
\end{equation}
As a symmetric monoidal functor preserves associative algebra objects, it is enough to show that $X \times_Y X \in \op{Ass}(\op{Corr}(\Ca_{/Y})_{\E_1,\E_2})$. As $Y$ is the final object in this category, the fiber product is usual in $\op{Corr}(\Ca)_{\E_1,\E_2}$.As $f$ and $\Delta_f$ are in $\E_1 \cap \E_2$, by \cref{dualizableobjinCorr} it follows that $X$ is dualizable in $\op{Corr}(\Ca)_{\E_1,\E_2}$. By \cite[Lemma B.1.16]{heyer20246functorformalismssmoothrepresentations}, we see that $\underline{End}(X,X)$ exists in $\op{Corr}(\Ca_{/Y})_{\E_1,\E_2}$ and it is equivalent to $X \times_{Y} X$. The proof concludes from the fact in a symmetric monoidal category $\D$, the internal Hom $\underline{Hom}(d,d)$ is an associative algebra object of $\D$.

\end{proof}

\begin{remark}
     The object $X \times_Y X$ being an associative algebra object, admits a map from the unit and multiplication which can be described as follows:
   \begin{enumerate}
       \item \underline{Unit map:} \begin{equation}
           \begin{tikzcd}
               * & X \arrow[l]\arrow[d,"\Delta_f"]\\
               {} & X \times_Y X.
           \end{tikzcd}
       \end{equation}
       \item \underline{Multiplicaltion Map:}
       \begin{equation}
           \begin{tikzcd}
               X \times_Y X  \times X \times_Y X & X \times_Y X \times_Y X \arrow[l,"\op{id} \times \Delta_f \times \op{id}"] \arrow[d,"\op{id} \times f \times \op{id}"] \\
               {} & X \times_Y X.
           \end{tikzcd}
       \end{equation}
   \end{enumerate}
\end{remark}

\subsection{Extension along geometric and exceptional pairs.}

In this section, we recall the extension of abstract six-functor formalisms to larger geometric setups. In particular, we recall the two specific extensions as discussed in \cite{chowdhury2025sixfunctorformalismsiiiconstruction}: geometric and exceptional pairs. Let us recall the definitions of these two notions.

\begin{definition}\label{Definition:Nice geometric and exceptional pairs.}
\begin{enumerate}
    \item An inclusion of two marked $\infty$-categories $(\Ca,\S,\E) \subset (\Ca',\S',\E')$ is a \textit{nice geometric pair} if the following conditions hold :
    \begin{enumerate}
        \item Each of the four pairs $(\Ca,\S),(\Ca,\E),(\Ca,\S')$ and $(\Ca',\E')$ are geometric setups.
        \item $\S' \cap \Ca_1 =\S$.
        \item For  $X' \in \Ca'$, there exists a morphism $x : X \to X'$ called an \textit{atlas} such that $X \in \Ca$ and for every $Y \to X'$ where $Y \in \Ca$, the base change $Y \times_{X'} X \to Y$ lies in $S$.
        \item For every $f : X' \to Y'$ in $\E'$ and for every atlas $y : Y \to Y'$, the base change morphism $X' \times_{Y'}Y \to Y$ is in $\E.$
    \end{enumerate}
    \item An inclusion of $2$-marked $\infty$-categories $(\Ca,\S,\E) \subset (\Ca,\S,\E')$ is an \textit{exceptional pair} if the following conditions are satisfied :
    \begin{enumerate}
        \item The pairs $(\C,\S),(\Ca,\E),(\Ca,\E')$ are geometric setups.
        \item $\S \subset \E$.
        \item For every $f: X \to Y$ in $\E'$, there exists a morphism of augmented simplicial objects 
    \begin{equation*}
        f_{\bb} : X_{\bb} \to Y_{\bb}
        \end{equation*}
        where 
        \begin{itemize}
            \item $f_{-1}=f$
            \item $f_n \in E$ for $n \ge 0$,
            \item $X_{\bb} \to X$ and $Y_{\bb} \to Y$ are $\S$-hypercovers.
        \end{itemize}
        
    \end{enumerate}
    \end{enumerate}
\end{definition}

    \begin{example}

    \begin{enumerate}
      \item Let $\Ca= \op{Sch}$ be the category of schemes and $\Ca'= \op{Algst}$ be the $(2,1)$-category of algebraic stacks. Let $\D$ be an abstract six-functor formalism which has descent along smooth surjective morphisms.  Considering $\S$ and $\S'$ as smooth surjections of schemes and algebraic stacks respectively along with $\E$ and $\E'$ be representable morphisms of locally of finite type of schemes and algebraic stacks respectively, we see that $(\Ca,\S,\E) \subset (\Ca',\S',\E')$ is a nice geometric pair.
         \item Let $\Ca$ be the category of schemes and $\E$ be the collection of morphisms that are separated and of finite type. Let $\D$ be a formalism which has descent along Zariski covers. Let $\S$ be the collection of Zariski covers. It follows from the definition that any morphism $f: X \to Y$ which is locally of finite type admits a morphism of augmeneted simplicial objects $f_{\bb}: X_{\bb} \to Y_{\bb}$ where $X_{\bb}$ and $Y_{\bb}$ are Zariski hypercovers and $f_n$ for $n \ge 0$ is separated and of finite type. Thus,, letting $\E'$ be the locally of finite type morphisms, shows that $(\Ca,\S,\E) \subset (\Ca,\S',\E')$ is an exceptional pair.
         \end{enumerate}
        \end{example}
The main theorem that involves extension along these pairs if the following:

\begin{theorem}\label{Theorem:extensionalonggeometricandexceptionalpairs}
    \begin{enumerate}
        \item Let $(\Ca,\S,\E) \subset (\Ca,\S',\E')$ be a nice geometric pair and let 
        \begin{equation}
            \D_{(\Ca,\E)} : \CrrCpEal \to \op{Cat}_{\infty}
        \end{equation}
        be an abstract six-functor formalism with the property that $\D^*$ satisfies descent along $\S$-\v{C}ech covers. Then $\D_{(\Ca,\E)}$ can be extended to an abstract six-functor formalism:
        \begin{equation}
            \D_{(\Ca',\E')} : \op{Corr}(\Ca')_{\E',\op{all}} \to \op{Cat}_{\infty}
        \end{equation}
        \item Let $(\Ca,\S,\E) \subset (\Ca,\S,\E')$ be an exceptional pair and let 
        \begin{equation}
            \D_{(\Ca,\E)} : \CrrCpEal \to \op{Cat}_{\infty}
        \end{equation}
        be an abstract six-functor formalism with the property that $\D_!$ satisfies codescent along $\S$-hypercovers. Then $\D_{(\Ca,\E)}$ can be extended to an abstract six-functor formalism:
        \begin{equation}
            \D_{(\Ca,\E')} : \op{Corr}(\Ca)_{\E',\op{all}} \to \op{Cat}_{\infty}
        \end{equation}
    \end{enumerate}
\end{theorem}

\section{Functorial Cohomological Purity}
In this section, we prove the functorial version of cohomological purity isomorphism $f_{\#} \cong  f_!\Sigma_f$. We at first construct the functorial version of $f_!\Sigma_f$ denoted by $\D^{\Sigma}_!$. This involves us diving into the technical construction of maps from multisimplicial nerves and proving \cite[Theorem 1.0.5]{chowdhury2025sixfunctorformalismsii} without the admissible edge condition. Once we have constructed the map, we use this to construct the natural transformation.

To every $f : X \to Y$ a morphism of schemes which is separated and locally of finite type we have the purity transformation functor :
\begin{equation}
    f_{\#} \to f_! \circ \Sigma_f
\end{equation}

The goal of this section is to construct a functor in the $\infty$-categorical setting 
\begin{equation}
    \D_{!}^{\Sigma} : \op{Sch} \to \op{Cat}_{\infty}
\end{equation}
which sends 
\begin{equation}
    f \longmapsto f_! \circ \Sigma_f
\end{equation}

The general theorem is the following : 

\begin{theorem} \label{puritysigmaequvialencefunctorial}
    Let $(\Ca,\E,\I,\P)$ be a Nagata setup and let 
    \begin{equation}
        \D_{(\Ca,\E)}:  \op{Corr}(\Ca)^{\otimes}_{\E,\op{all}} \to \op{Cat}_{\infty}
         \end{equation}
        be an abstract six-functor formalism with the following properties :
        \begin{enumerate}
            \item There exists a set of edges $\mathcal{S}$ stable under pullbacks and compositions containing $\I$ and contained in $\op{HL}$.
            \item The natural transformation $\op{Ex}_{\#!}$ is an equivalence for pullback squares.
            \end{enumerate}
        Then the assignment $ f \mapsto f_!\Sigma_f$ admits an enhancement to an $\infty$-functor 
        \begin{equation}
            \D^{\Sigma}_! : \Ca_{\E \cap \mathcal{S}} \to \op{Cat}_{\infty}
        \end{equation}
        which sends $X \mapsto \D(X)$
 and 
 \begin{equation}
     f \longmapsto f_!\Sigma_f \cong f_{\#}.
 \end{equation}
    
    \end{theorem}

Let us recall that for $f : X \to Y$, the functor  $\Sigma_f$ is defined as $\op{Th}(f,\delta_f)= \op{pr}_{1\#}\delta_{f!}$ where the maps are 
\begin{equation}
    X \xrightarrow{\delta_f} X \times_Y X \xrightarrow{\op{pr}_1} X
\end{equation}

The functor $f_!\Sigma_f$ arises from decomposing the commutative square :
\begin{equation}
    \begin{tikzcd}
        X \arrow[r,"\op{id}"] \arrow[d,"\op{id}"] & X \arrow[d,"f"] \\
        X \arrow[r,"f"] & Y
    \end{tikzcd}
\end{equation}

into the following diagram :
\begin{equation}
    \begin{tikzcd}
        X \arrow[dr,"\delta_f"] \arrow[drr,bend left=50,"\op{id}"] & {} & {} \\
        {} & X \times_Y X \arrow[d,"\op{pr}_2"] \arrow[r,"\op{pr}_1"] & X \arrow[d,"f"] \\
        {} & X \arrow[r,"f"] & Y 
    \end{tikzcd}
\end{equation}
One can realise this diagram as two pullback squares :
\begin{equation}
    \begin{tikzcd}
     X \arrow[d,"\delta_f"] \arrow[r,"\op{id}"]  & X \arrow[d,"\delta_f"] & {} \\
     X \times_Y X \arrow[r,"\op{id}"] & X \times_Y X \arrow[d,"\op{pr}_2"] \arrow[r,"\op{pr}_1"] & X \arrow[d,"f" ]\\
     {} & X \arrow[r, "f"] & Y
    \end{tikzcd}
\end{equation}
The functor $f_!\Sigma_f$ follows from applying the base change $\op{Ex}_{\#!}$ to both of the squares as a composition.  In order to compose these functors for an $n$-compoasble chain of morphisms, we need to use the simplicial techniques developed by Liu-Zheng in \cite{Gluerestnerv} and reproved in \cite{chowdhury2025sixfunctorformalismsii}. In particular, we shall use a variant of the arguments in   \cite[Section 5]{chowdhury2025sixfunctorformalismsii} where one deals with \textit{simplicial set of cartesianizations} $\Cart^n$.\\

We shall freely use the language of multi-simplicial sets (the tiled and marked versions) as introduced in \cite{liu2017enhanced} or \cite{chowdhury2025sixfunctorformalismsii}.\\

\textbf{First steps of the proof of \cref{puritysigmaequvialencefunctorial} :}
\begin{enumerate}
    \item We have a map 
    \begin{equation}
        \op{Sq}:\Ca_{\E \cap \mathcal{S}} \to \dd^*_2\Ca_{\mathcal{S},\E}
    \end{equation}
    where objects goes to same objects. For a $n$-composable chain of morphisms :
    \begin{equation}
      \tau:  X_0 \xrightarrow{f_0} X_1 \cdots X_{n-1} \xrightarrow{f_{n-1}} X_n
    \end{equation}
$\op{Sq}(\tau)$ is the following $n$ times $n$ grid :

\begin{equation}
    \begin{tikzcd}
        X_0 \arrow[r, "\op{id}"] \arrow[d,"\op{id}"] & X_0 \arrow[d,"f_0"] \arrow[r,"\cdots"] & 
       {} \arrow[r,"\cdots"] \arrow[d,"\vdots"]  &\arrow[r,"\cdots"] \arrow[d,"\vdots"] & X_0 \arrow[d,"f_0"] \\
        X_0 \arrow[d,"\vdots"] \arrow[r,"f_0"] & X_1 \arrow[r,"\cdots"]\arrow[d,"\vdots"] & {} \arrow[d,"\vdots"] \arrow[r,"\cdots"]&{} \arrow[d,"\vdots"] \arrow[r,"\cdots"] & X_1 \arrow[d,"\vdots"] \\
        {} \arrow[d,"\vdots"] \arrow[r,"\cdots"] & {} \arrow[d,"\vdots"] \arrow[r,"\cdots"] & {} \arrow[r,"\cdots"] \arrow[d,"\vdots"] & X_{n-1} \arrow[d,"\op{id}"] \arrow[r,"\op{id}"] & X_{n-1} \arrow[d,"f_{n-1}"] \\
        X_0 \arrow[r,"f_0"] & {} \arrow[r,"\cdots"] & {} \arrow[r,"\cdots"] & X_{n-1} \arrow[r,"f_{n-1}"] & X_n
    \end{tikzcd}
\end{equation}

\item
We shall use the following notation 
\begin{notation}\label{generalpullbackbasechangenotation}
Let $Q$ be the tiling of the two marked simplicial set $(\Ca,\{\S,\E\})$ given by the pullback squares for which $\op{Ex}_{\#!}$ is an equivalence. Let $R$ be the tiling of the same two marked simplicial set given by squares 
\begin{equation}
    \begin{tikzcd}
        X' \arrow[r]\arrow[d] & X \arrow[d]\\
        Y' \arrow[r] & Y
    \end{tikzcd}
\end{equation}
such that when decomposed, the pullback square in the diagram 
\begin{equation}
    \begin{tikzcd}
        X' \arrow[dr] & {} & {}\\
        {} & X'' \arrow[r] \arrow[d] & X \arrow[d] \\
        {} & Y' \arrow[r]  & Y
    \end{tikzcd}
\end{equation}
belongs to $Q$. Define
\begin{equation}
    \dd^*_2\Ca^{\op{cart},Q}_{\S,\E}:= \dd^*_2 \dd^{2*}_{\Box}(\Ca,\{(\S,\E,Q)\}) \,\, \dd^*_2\Ca^{R}_{\S,\E}:= \dd^*_2\dd^{2*}_{\Box}*(\Ca,\{(\S,\E,R)\}) .
\end{equation}
\item   Using the theorem of partial adjoints and the equivalence $\op{Ex}_{\#!}$, we have a following functor of the form:
\begin{equation}
\D^{\op{cart}}_{\#!} : \dd^*_2 \Ca^{Q,\op{cart}}_{\mathcal{S},\E} \to \op{Cat}_{\infty}
\end{equation}
which sends edges along $\mathcal{S}$ to $(-)_{\#}$ and the morphisms along $\E$ to $(-)_!$.
\item The above arguments give us the following diagram :
\begin{equation}
    \begin{tikzcd}
        {} & \dd^*_2 \Ca^{Q,\op{cart}}_{\mathcal{S},\E} \arrow[r,"\D_{\#!}"] \arrow[d,"p_{\op{cart}}"] & \op{Cat}_{\infty} \\
        \Ca_{\E \cap \S} \arrow[r,"\op{Sq}"] & \dd^*_2\Ca^R_{\mathcal{S},\E} \arrow[ur,dotted,"\D_{\#!}",swap] & {}
    \end{tikzcd}
\end{equation}
\end{notation}
\item \textbf{Goal:} Make sense of the dotted arrow $\D_{\#!}$ arising from $\D^{\op{cart}}_{\#!}$ such that $\D^{\op{cart}}_{\#!} \circ \op{Sq} = \D^{\Sigma}_!$. 
\item The existence of dotted is very similar to the statement in \cite[Theorem 1.5.1]{liu2017enhanced} which states one can extend functors to $\infty$-categories along $p_{\op{cart}}$ with a slight modification. The modification is that $\mathcal{S}$ is not admissible. 

\end{enumerate}
\subsection{ Extension along $p_{\op{cart}}$ for non-admissible edges.}
In this section, we prove the following general theorem :

  \begin{theorem}\label{Theorem: Ext along non adm edges}[Theorem B : Extension along $p_{\op{cart}}$ for non-admissible]\label{thmBvariant}
    		Let $\Ca$ be an $\infty$-category and $\E_1,\E_2$ be a collection of edges in $\Ca$ such that   $\E_1$ is admissible and $\E_2$ contains isomorphisms and is stable under pullbacks. Let $Q$ and $R$ be two collections of squares in $\Ca$ where:
            \begin{enumerate}
                \item $Q$ a subcollection of cartesian squares in $\Ca$ formed by $\E_1$ and $\E_2$ such that $(\Ca,\{\E_1,\E_2\},Q)$ form a $2$-tiled simplicial set.
                \item $R$ be a subcollection of commutative squares in $\Ca$ formed by $\E_1$ and $\E_2$ of the form \begin{equation}
    \begin{tikzcd}
        X' \arrow[r]\arrow[d] & X \arrow[d]\\
        Y' \arrow[r] & Y
    \end{tikzcd}
\end{equation}
such that when decomposed, the pullback square in the diagram 
\begin{equation}
    \begin{tikzcd}
        X' \arrow[dr] & {} & {}\\
        {} & X'' \arrow[r] \arrow[d] & X \arrow[d] \\
        {} & Y' \arrow[r]  & Y
    \end{tikzcd}
\end{equation}
belongs to $Q$.
            \end{enumerate}
             Then for any $\infty$-category $\D$, there exists a solution to the lifting problem:
\begin{equation}
\begin{tikzcd}
    \dd^*_2\Ca^{Q,\op{cart}}_{\E_1,\E_2} \arrow[d,"p_{\op{cart}}",swap] \arrow[r,"g_{\op{cart}}"] &\D\\
    \dd^*_2\Ca^R_{\E_1,\E_2} \arrow[ur,"g'_{\op{cart}}",swap,dotted] & {}.
    \end{tikzcd}
\end{equation}
where the notations are defined in \cref{generalpullbackbasechangenotation}.
\end{theorem}
\begin{remark}
The above theorem does not use the need of truncation condition on the classes of maps in $\E_1 \cap \E_2$ and also removes the admissibility condition of $\E_2$ (which was in \cite[Thm B]{chowdhury2025sixfunctorformalismsii}). Thus, this leads to revisiting the proof of \cite[Thm B]{chowdhury2025sixfunctorformalismsii} and making the necessary changes needed to circumvent the conditions to prove the theorem.
\end{remark}
The proof of the theorem really follows the same idea as \cite[ThmB]{chowdhury2025sixfunctorformalismsii}. Let us recall the key notations and definitions.

\begin{definition}
    \begin{enumerate}
        \item 	Let $P$ be a partially ordered set. $Q \subset P$ is said to be an \textit{up-set} if for every $q \in Q$ and $ p \ge q $ in $P$ implies $ p \in Q$. We shall denote the category of up-sets of $P$ by $\U(P)$. It is a partially ordered set where the ordering is given by inverse inclusion. 
        \item 	For any partially ordered set, we denote the products (infima) by $ \wedge$ and coproducts (suprema) by $\vee$.  In $\U(P)$, we have $ Q \wedge Q' = Q \cup Q'$ and $ Q \vee Q' = Q \cap Q'$. There is a canonical order preserving map $\varsigma^P: P \to \U(P)$ defined by $p \to P_{p/}$. 
        \item A square in a partially ordered set is an \textit{exact square} if it is both pushout and a pullback square.
        \item Consider $[n] \times [n]$. We shall denoted the partially ordered set of non-empty up-sets of $[n] \times [n]$ by $\Crt^n$. We denote  $\varsigma^n:= \varsigma^{[n] \times [n]}: [n] \times [n] \to \Crt^n$ to be the usual map sending $(p,q) \to ([n] \times [n])_{(p,q)/}$. Let $\Cart^n:= N(\Crt^n)$ and $\varsigma^n: \Delta^n \times \Delta^n \to \Cart^n$ be the map induced from $\varsigma^n$.  
        \item Let $\Ca$ be an $\infty$-category. Let $\tau: \Delta^n \times \Delta^n \to  \Ca$ be a map. We define the simplicial set $\Krt(\tau)$ which is defined as the pullback of the diagram:
		
		\begin{center}
			\begin{tikzcd}
				{} & \Krt(\tau)_{\op{RKE}} \arrow[d] \\
				\Delta^0 \arrow[r,"\tau"] & \op{Fun}(\Delta^n \times \Delta^n, \Ca) 
			\end{tikzcd}
		\end{center}	
		where $\Krt(\tau)_{\op{RKE}}$ is the sub-simplicial set of $\op{Fun}(\Cart^n,\Ca)$ which are right Kan extensions along $\varsigma^n$. 
    \end{enumerate}
\end{definition}

\begin{example}
    	 The diagram of $\Cart^1$ is as follows:
			\begin{center}
				\begin{tikzcd}
					b_{00} \arrow[dr] & {} & {} \\
					{} & P \arrow[r] \arrow[d] & b_{01} \arrow[d] \\
					{} & b_{10} \arrow[r] & b_{11}
				\end{tikzcd}
			\end{center}
			Here $b_{ij}:= ([1] \times [1])_{(i,j)/}$ and $P = b_{01} \wedge b_{10}$. 
\end{example}
\begin{proposition}
    \begin{enumerate}
        \item  Every morphisms $ Q \to Q'$ in $\U(P)$ is a composition of finite sequence of exact pullbacks of morphisms $ \sigma^P(x) \to \sigma^P(x)-x$ where $x \in Q -Q'$.
        \item 	Let $\Ca$ be an $\infty$-category and $F: N(\U(P)) \to \Ca$ be a functor. Then if $F$ is a right Kan extension along $\varsigma^P$, it sends exact squares to pullback squares. 
        \item If $\Ca$ admits pullbacks, the simplicial set $\Krt(\tau)$ is a contractible Kan complex. 

    \end{enumerate}
\end{proposition}

\begin{notation}
    \begin{enumerate}
			\item We have a map:\[\pi^n : \Crt^n \to [n] \times [n] \] defined as: 
			\[ \pi^n(P):= (\op{min}_{(p,q) \in P}p,\op{min}_{(p,q)\in P}q). \]
			\item $\pi^n \circ \varsigma^n = \op{id}_{[n]\times[n]}$.
			\item Other than $\varsigma^n$, we have two maps: \[ \xi^n,\eta^n: [n] \times [n] \to \Crt^n \] defined by 
			\[\xi^n(p,q):= \varsigma^n(p,0)\wedge \varsigma^n(0,q) ; \eta^n(p,q):= \varsigma^n(p,n)\wedge\varsigma^n(n,q).\]
			\item For $(p,q) \in [n] \times [n]$, we denote \[ \boxplus^n_{(p,q)}:= N(\Crt^n_{\xi^n(p,q)//\eta^n(p,q)}) \]
			\item Denote 
			\[ \boxplus^n:= \cup_{(p,q)\in [n] \times [n]} \boxplus^n_{(p,q)}. \]
			\item  For two elements $x\le y \in \op{Cart}^n$ and for $p,q \in [n]$, we define two elements:
         \begin{equation}
             \Lambda_p^n(x,y) := (\sigma^n(\pi^n_1(y) \vee p,0) \vee x) \wedge y \quad ; \quad \mu_q^n(x,y) := (\sigma^n(0,q \vee \pi^n_2(y)) \vee x) \wedge y.
         \end{equation}

		\end{enumerate}
\end{notation}

\begin{notation}
	Consider the bi-marked simplicial set $(\Delta^n \times \Delta^n, \F'_1:=(\epsilon^2_1 \Delta^{n,n})_1, \F'_2:=(\epsilon^2_2 \Delta^{n,n})_1)$. Let $(\Cart^n, \F_{\S})$ be the marked-simplicial set  $(\Cart^n, \F_{1,\S}:=(\pi^n)^{-1}(\F'_1)/X, \F_2:= (\pi^n)^{-1}(\F'_2))$. Here $X$ is the set of edges $f: x \to y , x \neq y$ such that $\pi^n(x) =\pi^n(y)$ \\
	We define $(\Cart^n,\F^{\op{exact}}_{\S})$ to the $2$-tiled simplicial set where the $2$-tiling is given by $\F_{12}:=\F_{1,\S} \star^{\op{cart}} \F_2$ as exact squares.
  
	\end{notation}

\begin{lemma}
    \begin{enumerate}
        \item $\Lambda_p^n(x,x)=x = \mu_q^n(x,x)$.
        \item $\pi^n(\Lambda^n_p(x,y))= (\pi^n_1(y),\pi^n_2(x))$ and $\pi^n(\mu^n_q(x,y)) = (\pi^n_1(x),\pi^n_2(y))$.
        \item The inclusion 
        \begin{equation}
        \gamma : \boxplus^n_{\op{cart}}:=\bigcup_{0 \le p \le n} \boxplus^n_{(p,n)} \hookrightarrow \Cart^n 
        \end{equation}
        is an inner anodyne.
    \end{enumerate}
\end{lemma}
\begin{notation}
    As the vertices of $\boxplus^n_{\op{cart}}$ are same as $\op{Cart}^n$, to every object $x \in \op{Cart}^n$, define
    \begin{equation}
        \alpha_x = \op{min}_{ x \in \boxplus^n_{(p,n)}} p
    \end{equation}
\end{notation}
\begin{lemma}\label{Lambdamulemmaequality}
    For $x\le y \in \Cart^n$ and $\pi^n(x)=\pi^n(y)$, 
    then $\mu^n_n(x,y)=y$.
\end{lemma}
\begin{proof}
We need to show that
\begin{equation}
    (\sigma^n(0,n) \vee x) \wedge y = y
\end{equation}
Let $(a,n) \in \sigma^n(0,n) \vee x$, then $\pi^n_1(y) = \pi^n_1(x) \le a$. By definition of $\pi^n$, it follows that there exists an element $(\pi^n_1(y),d) \in y$. As $d \le n$. Then we have 
\begin{equation}
    (\pi^n_1(y),d) \le (a,n)
\end{equation}
As $y$ is an up-set, this implies $(a,n) \in y$. Thus we have $\sigma^n(0,n) \vee x \ge y$. This proves that $(\sigma^n(0,n) \vee x) \wedge y = y$.
\end{proof}
\begin{lemma}
    Let 
    \begin{equation}
        \begin{tikzcd}
            x \arrow[r] \arrow[d] & \mu^n_q(x,y) \arrow[d] \\
            \Lambda^n_r(x,y) \arrow[r] & y
        \end{tikzcd}
    \end{equation}
    be a square in $\op{Cart}^n$ where $ x\le y \in \boxplus^n_{(p,n)}$. Then
    \begin{enumerate}
        \item If $r \le p$, the square is a pullback square.
        \item If $ r= p , q =n$, the square is an exact square.
    \end{enumerate}
    
\end{lemma}
\begin{proof}
\begin{enumerate}
    \item  We need to show $\Lambda^n_r(x,y) \wedge \mu^n_q(x,y)$ is $x$. By the property of pullback, we have $\Lambda^n_p(x,y) \wedge \mu^n_q(x,y) \ge x$. As $y \le \eta^n(p,n) =\varsigma^n(p,n)$, we have $\pi^n(y) \le (p,n)$. Thus $\sigma^n(r',q') \le x$ for $r' \le p$ and for all $q'$ ( as $q' \le n$). The pullback follows from the following chain of inequalities:
    \begin{align*}
        \Lambda^n_p(x,y) \wedge \mu^n_q(x,y)\\ = ((\varsigma^n(\op{max}(r,\pi^n_1(y)),0) \vee x) \wedge y) \wedge ((\varsigma^n(0,\op{max}(q,\pi^n_2(y))) \vee x) \wedge y) \\= (\varsigma^n(\op{max}(r,\pi^n_1(y)),0) \wedge \varsigma^n(0,\op{max}(q,\pi^n_2(y)))) \vee x \wedge y \\ \le x \vee x \wedge y = x
    \end{align*}
    \item By the previous point, the square is already a pullback square, we just need to show that it is a pushout square. Let $m = \op{max}(\pi^n_1(y),p)$. We have the following chain of inequalities. 
   \begin{align*}
       \Lambda^n_p(x,y) \vee \mu^n_n(x,y) \\= (\sigma^n(m,0) \vee \sigma^n(0,n)) \vee x \wedge y \\ = (\sigma^n(m,n) \vee x ) \wedge y  \\ \ge \sigma^n(p,n) \vee x \wedge y \\ \ge y \vee x \wedge y= y
   \end{align*}
   By pushout property, it follows that $\Lambda^n_p(x,y) \vee \mu^n_n(x,y) \le y$. Thus, combining, we get the equality proving the exactness of the square.
   \end{enumerate} 
\end{proof}

The main proposition is the following:
\begin{proposition}\label{Proposition: Key Prop for pcart non adm}
 There exists a  morphism 
\begin{equation}
    \epsilon_n^{\S} : \Cart^n \to \dd^*_2\dd^{2+}_*(\Cart^n,\F_{\S})
\end{equation}
such that :
\begin{enumerate}
    \item  $\epsilon_n^{\S}(f)$ for an edge $f: x \to y \in \boxplus^n_{\op{cart}}$ for $\pi^n(x)=\pi^n(y)$ is the exact square
    \begin{equation}
      \begin{tikzcd}
          x \arrow[d,"\op{id}"] \arrow[r,"f"] & x \arrow[d,"\op{id}"] \\
          y \arrow[r,"f"] & y
      \end{tikzcd}
      \end{equation}
      \item $\epsilon_n^{\S}$ restricts to a morphism :
      \begin{equation}
        \epsilon^{\op{cart},\S}_n: \boxplus^n_{\op{cart}} \to \dd^*_2\dd^{2\bx}_*(\Cart^n,\F^{\op{exact}}_{\S})
      \end{equation}
     which makes the following diagram commute
      \begin{equation}
          \begin{tikzcd}
               \boxplus^n_{\op{cart}}\arrow[r,"\epsilon_n^{\op{cart},\S}"] \arrow[d,hookrightarrow] \arrow[d,"\gamma"] & \dd^*_2\dd^{2\bx_*}(\Cart^n,\F_{\S}^{\op{exact}}) \arrow[d,hookrightarrow,"j"] \arrow[r,"\tau^n(\tau')"]& \dd^*_2\Ca^{Q,\op{cart}}_{\E_1,\E_2} \arrow[d,"p_{\op{cart}}"] \\
        \Cart^n \arrow[r,"\epsilon_n^{\S}"] & \dd^2_*\dd^{2+*}(\Cart^n,\F_{\S}) \arrow[r,"\tau^n(\tau')"] & \dd^*_2\Ca^R_{\E_1,\E_2}
          \end{tikzcd}
         \end{equation}
         Here $\tau$ is an $n$-simplex of $\dd^*_2\Ca_{\E_1,\E_2}$, $\tau'$ is an element of $\Krt(\tau)$ and morphisms are defined as in \cite[Proposition 5.1.15]{chowdhury2025sixfunctorformalismsii}.
         \item Moreover,  we also have the following commutative diagram:
         \begin{equation}
    \begin{tikzcd}
        \Krt(\tau) \arrow[r," \tau^n"] \arrow[dr,"\tau^n"] & \op{Fun}( \delta^*_2\delta^{2\bx}_*(\Cart^n,\F_{\S}^{\op{cart}}),\dd^*_2\Ca^{Q,\op{cart}}_{\E_1,\E_2}) \arrow[d,hookrightarrow]\\
        {} & \op{Fun}(\dd^*_2\dd^{2+}_*(\Cart^n,\F_{\S}),\dd^*_2\Ca^R_{\E_1,\E_2})
    \end{tikzcd}
\end{equation}
         
\end{enumerate}

\end{proposition}

\begin{proof}
    \textbf{Construction of $\epsilon_n^S$:\\}

    Let $\sigma: \Delta^m \to \Cart^n$ be a morphism. Let $x_0,x_1,\cdots x_m$ be vertices of this $m$-simplex. We define $\epsilon^S_n(\sigma): \Delta^m \times \Delta^m \to \Cart^n$ defined as follows:
    \begin{equation}
        \epsilon^S_n(\sigma)(a,b) = \begin{cases}
            \Lambda^n_{\alpha_{x_a}}(x_b,x_a) ~~ a\ge b\\
            \mu^n_n(x_a,x_b) ~~ a \le b
        \end{cases}
    \end{equation}
  We follow the steps of the proof of  \cite[Claim 5.5.2]{chowdhury2025sixfunctorformalismsii}, it turns out that we only need to check the condition for $a=b$, $f:x_a \to x_{a+1} \in \boxplus^n_{\op{cart}}$ and $\pi^n(x_a)=\pi^n(x_{a+1})$ . In this case, the square is :
  \begin{equation}
      \begin{tikzcd}
        x_a \arrow[r] \arrow[d] & \mu^n_n(x_a,x_{a+1}) \arrow[d]\\
        \Lambda^n_{\alpha_{x_{a+1}}}(x_a,x_{a+1}) \arrow[r] & x_{a+1}
      \end{tikzcd}
  \end{equation}
  By \cref{Lambdamulemmaequality}, we get that the square looks like this :
  \begin{equation}
       \begin{tikzcd}
          x_a \arrow[d,equal] \arrow[r,"f"] & x_{a+1} \arrow[d,equal] \\
        x_a \arrow[r,"f"] & x_{a+1}
            \end{tikzcd}
  \end{equation}
This shows that $\epsilon^{\S}_n$ satisfies Condition $1$. Conditions $2$ and $3$ just exactly follows as the arguments in \cite{chowdhury2025sixfunctorformalismsii}. Let us roughly state the changes in the argument. As $\tau'$ is an element of $\Krt(\tau)$, it is a right Kan extension of $\tau$ along $\sigma^n : [n] \times [n] \to \Cart^n$.  By \cite[Lemma 5.1.4]{chowdhury2025sixfunctorformalismsii}, the map $\tau^n(\tau')$ sends exact squares to pullback squares. This completes the proof of 2.

\end{proof}

\subsection{Proof of \cref{Theorem: Ext along non adm edges}}

As the proof of \cref{Theorem: Ext along non adm edges} follows the same procedure as the proof of \cite[Thm B]{chowdhury2025sixfunctorformalismsii}, we just give a sketch of the main points of the proof.

\begin{proof}[Proof of \cref{Theorem: Ext along non adm edges}]
  Let $\tau$ be an $n$-simplex of $\dd^*_2\Ca^r_{\E_1,\E_2}$.   Using Point 2 of \cref{Proposition: Key Prop for pcart non adm}  and using variant of \cite{chowdhury2025sixfunctorformalismsii}, we get the map
 \begin{equation}
     \alpha'_n : \Krt(\tau) \xrightarrow{\tau^n}  \op{Fun}( \delta^*_2\delta^{2\bx}_*(\Cart^n,\F_{\S}^{\op{cart}}),\dd^*_2\Ca^{Q,\op{cart}}_{\E_1,\E_2}) \xrightarrow{g_{\op{cart}}} \op{Fun}( \delta^*_2\delta^{2\bx}_*(\Cart^n,\F_{\S}^{\op{cart}}),\D) \xrightarrow{\epsilon_n^{\op{cart},\S}} \op{Fun}(\boxplus^n_{\op{cart}},\D)
 \end{equation}
 In a similar way as in \cite[Thm B]{chowdhury2025sixfunctorformalismsii}, we define 
 \begin{equation}
     \N(\tau):= \Krt(\tau) \times_{\op{Fun}(\boxplus^n_{\op{cart}},\D)} \op{Fun}(\Cart^n,\D)
 \end{equation}
 As $\Krt(\tau)$ is weakly contractible and $\boxplus^n_{\op{cart}} \hookrightarrow \Cart^n$ is inner anodyne, we have $\N(\tau)$ is weakly contractible. We also get the map:
 \begin{equation}
     \alpha^n : N(\tau) \to \op{Fun}(\Cart^n,\D) \xrightarrow{i \circ \sigma^n} \op{Fun}(\Delta^n,\D).
 \end{equation}
Again following the exact argument in \cite[Thm B]{chowdhury2025sixfunctorformalismsii}, we see that $\alpha^n$ constructed is functorial over $\Delta_{/\dd^*_2\Ca_{\E_1,\E_2}}$ and it satisfies the condition of \cite[Theorem 4.1.1]{chowdhury2023sixfunctorformalismsi}. This produces the map:
\begin{equation}
    g_{\op{comm}} : \dd^*_2\Ca^R_{\E_1,\E_2} \to \D.
\end{equation}
\end{proof}

\subsection{The natural transformation $\op{Pur}_{\D}$}

The theorem below upgrades the equivalence $f_! \circ \Sigma_f \cong f_{\#}$.

\begin{theorem}\label{Theorem:PurityFunctorialFullForm}
Under the conditions of \cref{puritysigmaequvialencefunctorial}, there exists a natural transformation:
\begin{equation}
   \op{Pur}_{\D}: \Delta^1 \times \Ca_{\E \cap \S} \to \op{Cat}_{\infty}
\end{equation}
where 
\begin{enumerate}
    \item $\op{Pur}_{\D}|_{[0] \times \Ca} = \D^{\Sigma}_!$.
    \item $\op{Pur}_{\D}|_{[1] \times \Ca} = \D_{\#}$.
    \item It is an equivalence.
\end{enumerate}
\end{theorem}
\begin{proof}
 We have the functor :
 \begin{equation}
     \D_{\#!} : \dd^*_2\Ca^R_{\S,\E} \to \op{Cat}_{\infty}
 \end{equation}
  as constructed before. The theorem shall follow from constructing a map:
  \begin{equation}
      \op{Pur}'_{\D} : \Delta^1 \times \Ca_{\S}\to \dd^*_2\Ca^R_{\S,\E}
  \end{equation}
  which after composing with $\D_{\#!}$ satisfies the condition. By adjunction, we construct the morphism
  \begin{equation}
      \op{Pur}'_{\D}: \Ca_{\E \cap \S} \to \op{Fun}(\Delta^1,\dd^*_2\Ca^R_{\S,\E}).
  \end{equation}
The construction is defined on level of every $n$-simplices:
\begin{enumerate}
    \item \textbf{n=0:} Let $X \in \Ca$, we define 
    \begin{equation}
        \op{Pur}'_{\D}(x):= \begin{tikzcd}
            X \arrow[r,equal]  \arrow[d,equal] & X \arrow[d,equal] \\
            X \arrow[r,equal] & X
        \end{tikzcd}
    \end{equation}
    \item \textbf{n=1:} Let $f: X \to Y \in \Ca_{\E \cap \S}$, then $\op{Pur}'_{\D}(f)$ should be a morphism of the form
    \begin{equation}
        \Delta^1 \times \Delta^1 \to \dd^*_2\Ca_{\S,\E}
    \end{equation}

Such a morphism comprises of two simplicies with the same edge. The two simplicies are defined as follows:
\begin{equation}
    \begin{tikzcd}
       X \arrow[r,equal] \arrow[d,equal] & X \arrow[r,equal] \arrow[d,"f"] & X \arrow[d,"f"] \\
       X \arrow[r,"f"] \arrow[d,equal] & Y \arrow[r,equal] \arrow[d,equal] & Y \arrow[d,equal] \\
       X \arrow[r,"f"] & Y \arrow[r,equal]& Y
    \end{tikzcd} 
\end{equation}
and 
\begin{equation}
\begin{tikzcd}
    X \arrow[r,equal] \arrow[d,equal] & X \arrow[r,equal] \arrow[d,equal] & X \arrow[d,equal]\\
    X \arrow[r,equal] \arrow[d,equal] & X \arrow[d,"f"] \arrow[r,equal] & X \arrow[d,"f"] \\
    X \arrow[r,"f"] & Y \arrow[r,equal] & Y 
    \end{tikzcd}
\end{equation}
\item \textbf{n=m:} Let $\sigma : \Delta^m \to \Ca$ and we denote the vertices by $x_0,x_1,\cdots x_m$. 
We need to define a morphism:
\begin{equation}
    \Delta^m \times \Delta^1 \to \dd^*_2\Ca_{\E_1,\E_2}
\end{equation}
Note that to define such a morphism, we need to define the $(m+1)$ non-degenerate $\Delta^{m+1} \to \Delta^m \times \Delta^1$ which are glued together along the faces. Such a simplex is enumerated by vertices $\sigma_i:\{(0,1),(0,2) \cdots (0,i),(1,i),(1,i+1),\cdots (1,m)\}$ for $0 \le i \le m$.
We define $\op{Pur}(\sigma)(\sigma_i): \Delta^{m+1} \to \dd^*_2\Ca_{\S,\E}$ as a morphism
\begin{equation}
    \Delta^{m+1} \times \Delta^{m+1} \to \Ca
\end{equation}
Using the fact that inclusion $I^{m+1} \hookrightarrow \Delta^{m+1}$ is an inner anodyne and $I^{m+1} = \coprod_{j=1}^{m} \Delta^{\{j-1,j\}}$, it is sufficient to define
\begin{equation}
    \op{Pur}'_{\D}(\sigma)(\sigma_i)_{a,b}:\Delta^{\{a,a+1\}} \times \Delta^{\{b,b+1\}} \to \Ca
\end{equation}
for $0 \leq a,b \leq m$. We define the square on the range of $a$ and $b$ as follows:
\begin{enumerate}
    \item \textbf{a=b:} 
   \begin{equation}
       \op{Pur}'_{\D}(\sigma)(\sigma_i)_{a,a}= \begin{cases}
           \begin{tikzcd}
               x_a \arrow[r,equal] \arrow[d,equal] & x_a \arrow[d] \\
               x_a \arrow[r] & x_{a+1}
           \end{tikzcd} \, a<i \\ \,\,
            \begin{tikzcd}
               x_i \arrow[r,equal] \arrow[d,equal] & x_i \arrow[d,equal] \\
               x_i \arrow[r,equal] & x_i
           \end{tikzcd} \,\,\,\,\,\,\,\, a=i \\
           \begin{tikzcd}
               x_a \arrow[r,equal] \arrow[d] & x_i \arrow[d] \\
               x_{a+1} \arrow[r,equal] & x_{a+1}
           \end{tikzcd}\, a > i \\
           \end{cases}
   \end{equation}
   \item \textbf{$a>b$:}
   \begin{equation}
       \op{Pur}'_{\D}(\sigma)(\sigma_i)_{a,b}=\begin{cases}
           \begin{tikzcd}
               x_b \arrow[r,equal] \arrow[d] & x_b \arrow[d] \\
               x_{b+1} \arrow[r,equal] & x_{b+1}
           \end{tikzcd} \, b<i \\ \,\,
           \begin{tikzcd}
               x_i \arrow[r,equal] \arrow[d,equal] & x_i \arrow[d,equal] \\
               x_i \arrow[r,equal] & x_i
           \end{tikzcd}\,\,\,\,\,\,\, b=i \\
          \begin{tikzcd}
               x_b \arrow[r,equal] \arrow[d] & x_b \arrow[d] \\
               x_{b+1} \arrow[r,equal] & x_{b+1}
           \end{tikzcd}\, b>i
       \end{cases}
       \end{equation}
       \item \textbf{$a>b$:}
        \begin{equation}
       \op{Pur}'_{\D}(\sigma)(\sigma_i)_{a,b}=\begin{cases}
              \begin{tikzcd}
               x_a \arrow[r] \arrow[d,equal] & x_{a+1} \arrow[d,equal] \\
               x_{a} \arrow[r] & x_{a+1}
           \end{tikzcd} \, b\leq i \\ 
            \begin{tikzcd}
               x_i \arrow[r,equal] \arrow[d,equal] & x_{i} \arrow[d] \\
               x_i \arrow[r] & x_{i+1}
           \end{tikzcd} \,\,\, b=i+1,a=i\\
           \begin{tikzcd}
               x_a \arrow[r] \arrow[d,equal] & x_{a+1} \arrow[d,equal] \\
               x_{a} \arrow[r] & x_{a+1}
           \end{tikzcd} \, \text{otherwise}
       \end{cases}
   \end{equation}
   
\end{enumerate}
By construction, these simplices glue together to form $\Delta^m \times \Delta^1 \to \dd^*_2\Ca_{\S,\E}$. This completes the construction of $\op{Pur}'_{\D}$ and hence $\op{Pur}_{\D}$. The construction of $\op{Pur}'_{\D}$ is done in such a way that it satisfies the condition of the proposition. 
\end{enumerate}
    
\end{proof}

\begin{remark}\label{remark:extendingpuritytocat2}
We expect the functor $\D^{\Sigma}_!$ to exist even if $\op{Ex}_{\#!}$ is not an equivalence. For this to make sense, the target is the $(\infty,2)$-category of $(\infty,1)$-categories $\op{Cat}_{(\infty,2)}$. Once this is achieved, one can construct $\op{Pur}^{\#,\Sigma}_{\D}$ also evaluated in $\op{Cat}_{(\infty,2)}$ which agrees to our construction once $\op{Ex}_{\#!}$ is an equivalence.
    
\end{remark}
\section{Extension of 6ff to Ind- and Pro-categories.}
In the first part of this section, we introduce the setup and the specific functor categories related to the geometric setup.. We restrict ourselves to considering specific morphisms as transition maps for these functors. This is motivated by projective and inductive systems, such as the classifying space of the loop group and the affine Grassmannian. This helps us to extend the six-functor formalism. We also prove some preliminary propositions that help us to extend the six-functor formalism in the next subsection.

\subsection{Setup and Definitions} Let $(\Ca,\E,\I,\P)$ be a Nagata Setup. Let 

\begin{equation}
    \D^{\op{op}} : \Ca^{\op{op}} \to \op{CAlg}(\op{Pr}^{L}_{\op{cl}})
\end{equation}
which satisfies the conditions of \cite[Theorem 5.2.3]{chowdhury2025sixfunctorformalismsiiiconstruction} so that it can be upgraded to a six-functor formalism 
\begin{equation}
    \D : \op{Corr}(\Ca)_{\E,\op{all}} \to \op{Pr}^L_{\op{cl}}.
\end{equation}
\begin{definition}
    \begin{enumerate}
        \item  Let $\op{HL}$ be the set of class of maps in $f \in \Ca$ such that $f^*$ admits a left adjoint $f_{\#}$ which satisfies base change and projection formula with respect to $(-)^*$ (\cite[Remark 8.27]{zhu2025tamecategoricallocallanglands}). 
      
        \item A class of morphisms $\S \in \E \cap \op{HL}$ is said to satisfy \textit{cohomological purity} if 
        \begin{enumerate}
            \item The natural transformation constructed from the following commutative diagram 

            \begin{equation}\label{diagaonalofmorphismdiag}
                \begin{tikzcd}
                    X \arrow[dr,"\Delta_f"] \arrow[ddr,bend right=80,"\op{id}"] \arrow[drr, bend left =80, "\op{id}"] & {} & {}\\
                    {} & X \times_Y X \arrow[r,"\op{pr}_1"] \arrow[d,"\op{pr}_2"] & X \arrow[d,"f"] \\
                    {} & X \arrow[r,"f"] & Y
                      \end{tikzcd}
            \end{equation}
               given by 
        \begin{equation}
            \op{Pur}_f : f_{\#}=f_{\#}(\op{pr}_{1!}\Delta_{f!}) \xrightarrow{\op{Ex}_{\#!}} f_! (\op{pr}_{2\#}\Delta_{f!}) = f_!\circ \Sigma_f
        \end{equation}
        is an equivalence.
        \item The endofunctor $\Sigma_f:=pr_{2\#}\Delta_{f!}$ is an equivalence.
        \end{enumerate}
 \item A map $f \in \E$ is said to be \textit{cohomologically proper} if $f_! = f_*$.    
  \end{enumerate}
\end{definition}

\begin{remark}
\begin{enumerate}
    \item Let $f \in \E$ such that $\Delta_f$ is cohomologically proper. Then one can construct the natural transformation $\alpha_f : f_! \to f_*$. Considering the same commutative diagram (\cref{diagaonalofmorphismdiag}), we have series of maps:

\begin{equation}
    f_! = f_!(\op{pr_2}_*\Delta_{f*}) \xrightarrow{\op{Ex}_{!*}} f_*\op{pr}_{2!}\Delta_{f*} \xrightarrow[\Delta_f~\text{coh.prop}]{\op{\sim}}f_*\op{pr}_{2!}\Delta_{f!} \xrightarrow{\op{\simeq}} f_*.
\end{equation}
Note that if $f$ is cohomologically proper, then $\alpha_f$ is an equivalence. For example if $f$ is in $\P$, by construction of $(-)_!$  glued from morphisms in $\I$ and $\P$, we have $f$ is cohomologically proper. 
\item If $\op{Ex}_{\#!}$ is an equivalence, then for $f \in \op{HL}$, $\op{Pur}_f$ is an equivalence . 

\end{enumerate}
\end{remark}
\begin{notation}\label{corrsmoothcategory}
 Let 
 \begin{equation}
    \D^{\op{op}} : \Ca^{\op{op}} \to \op{CAlg}(\op{Pr}^{L}_{\op{cl}})
\end{equation}
which satisfies the conditions of \cite[Theorem 5.2.3]{chowdhury2025sixfunctorformalismsiiiconstruction} so that it can be upgraded to a six-functor formalism 
\begin{equation}
    \D : \op{Corr}(\Ca)_{\E,\op{all}} \to \op{Pr}^L_{\op{cl}}.
\end{equation}
Let $\S \in \E \cap \op{HL}$ which satisfy cohomological purity. \\
   Let  $\delta^*_{3,\{2,3\}}\Ca^{\op{cart}^*_!}_{\E,\S\op{all}}$ be the subsimplicial set  of  $ \tau :\dd^*_{3,\{2,3\}}\Ca_{\E,\S,\op{all}}$ whose $n$-simplices are maps $\Delta^n_{\E} \times \times (\Delta^n)^{\op{op}}_{\S}\times (\Delta^n)_{\op{all}}^{\op{op}} \to \Ca$ such that for all $\alpha:\Delta^1_{\E} \times (\Delta^1)^{\op{op}}_{\S}\times (\Delta^1)_{\op{all}}^{\op{op}}  \to  \Delta^n_{\E} \times \times (\Delta^n)^{\op{op}}_{\S}\times (\Delta^n)_{\op{all}}^{\op{op}} $, the composite  \begin{equation} \tau \circ \alpha : \Delta^1_{\E} \times (\Delta^1)^{\op{op}}_{\S}\times (\Delta^1)_{\op{all}}^{\op{op}} \to \Ca \end{equation} of the form :
\begin{equation}\label{1simplexforDcartuppershriek}
    \begin{tikzcd}
        X_0 \arrow[rr," f_3"] \arrow[dr,"g_3"] \arrow[dd,"h_3"] && Y_0 \arrow[dd,"h_2"{xshift=6pt,yshift=-10pt}] \arrow[dr,"g_2"] & {} \\
        {} & X_1 \arrow[rr,"f_2" {xshift=-5pt},crossing over] \arrow[dd,"h_1" {xshift =3pt, yshift=6pt}] && Y_1 \arrow[dd,"h_0"] \\ 
        Z_0 \arrow[rr,"f_1"{xshift=-10pt,yshift=-1pt}] \arrow[dr,"g_1"] && W_0 \arrow[dr,"g_0"] & {} \\
        {} & Z_1  \arrow[rr,"f_0"] \arrow[from=uu,crossing over] && W_1 
    \end{tikzcd}
\end{equation}
where $f_i \in \E, g_i \in \S, h_i \in \op{All}$ for all $i \in \{1,2,3\}$  
 such that
 \begin{enumerate}
 \item the squares
 \begin{equation}
     \begin{tikzcd}
         X_0 \arrow[r,"f_3"] \arrow[d,"h_3"] & Y_0 \arrow[d,"h_2"] \\
         Z_0 \arrow[r,"f_1"] & W_0
     \end{tikzcd}
     and 
      \begin{tikzcd}
         X_1 \arrow[r,"f_2"] \arrow[d,"h_1"] & Y_1 \arrow[d,"h_0"] \\
         Z_1 \arrow[r,"f_0"] & W_1
     \end{tikzcd}
 \end{equation}
 are Cartesian. 
\item On applying $\D$, the squares:
 \begin{equation}\label{commutativesquaresintriplesimplicialset}
     \begin{tikzcd}
         \D(X_0)  \arrow[d,"f_{3!}"] & \D(X_1) \arrow[d,"f_{2!}"]\arrow[l,"g_{3}^*"] \\
         \D(Y_0)  & \D(Y_1) \arrow[l,"g_{2}^*"]
     \end{tikzcd}
     \text{and}
     \begin{tikzcd}
           \D(Z_0)\arrow[d,"f_{1!}"] & \D(Z_1) \arrow[d,"f_{0!}"]  \arrow[l,"g_{1}^*"] \\
         \D(W_0)  & \D(W_1) \arrow[l,"g_{0}^*"]
     \end{tikzcd}
 \end{equation}

 commute.

One can similarly define $\dd^*_{3,\{2,3\}}\Ca^{\op{cart}^!_!}_{\E,\S,\op{all}}$ in the similar fashion where we replace $(-)^*$ in horizontal directions by $(-)^!$ in \cref{commutativesquaresintriplesimplicialset}.
 \end{enumerate}
\end{notation}
\begin{proposition}\label{Proposition: corrextendedformalismincludingsmooth}
    Let $(\Ca,\E,\I,\P)$ be a Nagata setup. Let 

\begin{equation}
    \D^{\op{op}} : \Ca^{\op{op}} \to \op{CAlg}(\op{Pr}^{L}_{\op{cl}})
\end{equation}
which satisfies the conditions of \cite[Theorem 5.2.3]{chowdhury2025sixfunctorformalismsiiiconstruction} so that it can be upgraded to a six-functor formalism 
\begin{equation}
    \D : \op{Corr}(\Ca)_{\E,\op{all}} \to \op{Pr}^L_{\op{cl}}.
\end{equation}
Let $\S \subset \op{HL}$ be a weakly stable class of maps. 
Then $\D_{(\Ca,\E)}$ can be extend to a lax symmetric monoidal functor :
\begin{equation}
    \D_{(\Ca,\S,\E)} : \delta^*_{3,\{2,3\}}((\Ca^{\op{op}})^{\coprod\op{op}})^{\op{cart}^*_!}_{\E,\S\op{all}} \to \op{Pr}^L_{\op{cl}}
\end{equation}
 such that $\D{_(\Ca,\S,\E)}$ sends a $1$-simplex of the form \cref{1simplexforDcartuppershriek} to the morphism:
\begin{equation}
    g_3^* \circ f_{2!} \circ h_0^*: \D(W_1) \to \D(X_0) \end{equation}
\end{proposition}

\begin{proof}
  The proof is just a variant of the construction of $\D_{(\Ca,\E)}$ (\cref{theorem: mainconstructiontheorem}) from $\D$ as proved in \cite[Theorem 5.2.3]{chowdhury2025sixfunctorformalismsiiiconstruction}.  

Let $\op{\Box^{\S}}$ be the tiling on the simplicial set $\Ca$ given by $(\I,\P,\S,\op{all},\{ \mathcal{Q}_{ij} \}_{1\le i,j \le 4})$ where
:
\begin{enumerate}
    \item $Q_{ij}$ are cartesian squares for $(i,j)=\{(1,2),(1,4),(2,4)\}$.
    \item $Q_{ij}$ are commutative squares for $(i,j)=\{(3,4)\}$.
    \item $Q_{ij}$ for $(i,j)=\{(1,3),(2,3)\}$ are squares of the form 
    \begin{equation}
    \begin{tikzcd}
        X' \arrow[r,"f'"] \arrow[d,"g'"] & Y'\arrow[d,"g"] \\
        X \arrow[r,"f"] & Y
        \end{tikzcd}
    \end{equation}
    such that 
    \begin{equation}
    \begin{tikzcd}
        \D(X') \arrow[r,"f'_!"] & \D(Y') \\
        \D(X) \arrow[u,"g^{'*}"] \arrow[r,"f_!"] & \D(Y)\arrow[u,"g^*"]
        \end{tikzcd}
    \end{equation}
    commutes.
\end{enumerate}

The pullback functor $\Ca^{\op{op}} \to \op{CAlg}(\op{Pr}^L)$ produces the functor 
\begin{equation}
   \D_{1,\S}: \dd^*_{4,\{1,2,3,4\}}(\Ca^{\op{op}})^{\coprod,\op{op}})^{\Box^{\S}}_{\I,\P,\S,\op{all}} \to \op{Cat}_{\infty}
\end{equation}
where all the directions map to $(-)^*$.\\
Following the arguments of \cite[Theorem 5.2.3]{chowdhury2025sixfunctorformalismsiiiconstruction} using the theorem of Partial Adjoints(\cite[Theorem 3.2.1]{chowdhury2025sixfunctorformalismsiiiconstruction}) and the Compactification Theorem (\cite[Theorem 2.2.4]{chowdhury2025sixfunctorformalismsii}, we get the desired functor

\begin{equation}
    \D_{(\Ca,\S,\E)} : \dd^*_{3,\{2,3\}}((\Ca^{\op{op}})^{\coprod,\op{op}})^{\op{cart}^*_!}_{\E,\S,\op{all}} \to \op{Pr}^L_{\op{cl}}
\end{equation}

\end{proof}

\begin{proposition}\label{Proposition: corrformalismextendingproper}
      Let $(\Ca,\E,\I,\P)$ be a Nagata setup. Let 

\begin{equation}
    \D^{\op{op}} : \Ca^{\op{op}} \to \op{CAlg}(\op{Pr}^{L}_{\op{cl}})
\end{equation}
which satisfies the conditions of \cite[Theorem 5.2.3]{chowdhury2025sixfunctorformalismsiiiconstruction} so that it can be upgraded to a six-functor formalism 
\begin{equation}
    \D : \op{Corr}(\Ca)_{\E,\op{all}} \to \op{Pr}^L_{\op{cl}}.
\end{equation}

Then $\D$ can be extend to a lax symmetric monoidal functor :
\begin{equation}
    \D_{(\Ca,\P,\E)} : \delta^*_{3,\{3\}}((\Ca^{\op{op}})^{\coprod\op{op}})^{\op{cart}^*_*}_{\E,\P,\op{all}} \to \op{Pr}^L_{\op{cl}}
\end{equation}
 such that $\D{(\Ca,\P,\E)}$ sends a $1$-simplex of the form \cref{1simplexforDcartuppershriek} to the morphism:
\begin{equation}
    (g_3)_* \circ f_{2!} \circ h_0^*: \D(W_1) \to \D(X_0) \end{equation}
\end{proposition}
\begin{proof}
    The proof follows the same arguments as the previous proposition (\cref{Proposition: corrextendedformalismincludingsmooth}) using a different tiling including maps in the direction of $\P$ with a suitable commutative square hypothesis. 
\end{proof}

\begin{definition}\label{definition:IndProcatdefinition}
    Let $K$ be a filtered $\infty$-category. According to the setup above, we define the following categories:
    \begin{enumerate}
        \item Let $\op{Fun}^K_S(\Ca^{\op{op}})$ be the full subcategory of $\op{Fun}(K,\Ca^{\op{op}})$ spanned by functors $f: K \to \Ca^{\op{op}}$ such that $f(e)\in S$ for all edges $e \in K$. 
         \item Let $\tilde{\P} \subseteq \P$ be a subclass of morphisms in $\E$ stable under pullbacks, compositions and contains isomorphisms. Let $\op{Ind}^K_{\tilde{\P}}(\Ca)$ be the full subcategory of $\op{Fun}(K,\Ca)$ spanned by functors $f: K \to \Ca$ such that $f(e) \in \E$ for all edges $e \in K$. Similarly as in the previous situation, let $\op{Ind}(\E)$ be the set of edges in $\op{Ind}^K_{\tilde{\P}}(\Ca)$ which are levelwise in $\E$.
    \end{enumerate}
\end{definition}

\begin{remark}
  The above definition is motivated by the fact that for  the pro-algebraic group scheme $L^+G$ has smooth surjections with transition maps, and the ind scheme $LG$ has closed immersions (i.e., proper) as transition maps. The objects of  $\op{Pro}^K_{\S}(\Ca)$ can be related to placid schemes as defined in \cite{BOUTHIER2022108572} when $K$ is filtered.
\end{remark}
We define a special class of morphisms in Ind and Pro categories that will be useful for stating the theorems. 
\begin{definition}\label{defintiion:IndProAdjedgesdefinition}

\begin{enumerate}
    \item Let $\op{Adj}_!^!(\E))$ be the set of edges $f \in \Delta^1 \times K \to \Ca_{\E}$ in $\op{Pro}^K_{\S}(\Ca)$  with the following conditions: 
    \begin{enumerate}
    \item they are stable under pullbacks in $\op{Pro}^K_{\S}(\Ca)$.
    \item For every $\tau: k \to k'$, the square  
    \begin{equation}
    \begin{tikzcd}
        \D(f(0,k')) \arrow[r," p_!"] \arrow[d,"f_{k!}"] & \D(f(0,k)) \arrow[d,"f_{k'!}"]\\
        \D(f(1,k')) \arrow[r,"q_!"] & \D(f(1,k))
        \end{tikzcd}
    \end{equation}
    is horizontally right adjointable.
    \end{enumerate}
    \item Let $\op{Adj}_!^*(\op{All})$ be the set of edges $f \in \Delta^1 \times K^{\op{op}} \to \Ca$ in $\op{Ind}^K_{\tilde{\E}}\Ca$  with the following conditions:

    \begin{enumerate}
    \item they are stable under pullbacks in $\op{Ind}^K_{\tilde{\P}}(\Ca)$.
    \item For every $\tau: k \to k'$, the square  given by
    \begin{equation}
    \begin{tikzcd}
        \D(f(1,k)) \arrow[r," p_!"] \arrow[d,"f_{k}^*"] & \D(f(1,k')) \arrow[d,"f_{k'}^*"]\\
        \D(f(0,k)) \arrow[r,"q_!"] & \D(f(0,k'))
        \end{tikzcd}
    \end{equation}
    commutes. 
    \end{enumerate}
\end{enumerate}

\end{definition}
\begin{example}\label{example:proandindadjointablemorphisms}
    \begin{enumerate}
        \item Assume for every morphism in $\S$, the pullback functor is fully faithful and $\D$ satisfies cohomological purity i.e. $f^! \cong f^*$. Then every edge $ f \in \op{Pro}^K_{\S}(\Ca)$ between the pro system $X_{\bb}:= f|_{[0]}$ and the constant system $f_{[1]}= Y_{\bb}= Y$ belongs to $\op{Adj}^!(\E)$\footnote{this is essentially the argument in \cite[Example 2.6]{yaylali2025rationalmotivesproalgebraicstacks}}. Firstly, we check the condition of horizontal right adjointability. This follows from the fact that for every $k' \to k$ in $\Ca$, the diagram 
        \begin{equation}
        \begin{tikzcd}
            \D(X_k) \arrow[dr,"f_{k!}"] && \D(X_{k'}) \arrow[ll,"p^!"] \arrow[dl,"f_{k'!}"] \\
            {} & \D(Y) & {}
            \end{tikzcd}
        \end{equation}
        commutes. This follows from the chain of isomorphisms:
       \begin{align*}
           f_{k!} p^! \\ \cong  f_{k'!}p_!p^! \\ \xrightarrow[\op{Coh. Purity}]{\cong} f_{k'!}p_{\#}p^* \\ \xrightarrow[p^*\,\text{fully faithful}]{\cong} f_{k'!}
       \end{align*}
       Secondly, we need to show that pullback of such a square is stable under pullbacks. Let us take the pullback of $f$ along $g: Z_{\bb} \to Y$. We need to show that for any edge $k' \to k$ in $K$, the diagram:
       \begin{equation}
            \begin{tikzcd}
            \D(X'_k:= X_k \times_Y Z_k) \arrow[d,"f'_{k!}"] & \D(X'_{k'}:=X_{k'} \times_Y Z_{k'}) \arrow[l,"p^{'!}"] \arrow[d,"f'_{k'!}"] \\
            \D(Z_k) &  \D(Z_{k'}) \arrow[l,"q^!"]
            \end{tikzcd} 
       \end{equation}
       commutes. This follows from the previous chain of arguments applied to this square as $p^{'*}$ is fully-faithful($p'$ is pullback of $p$) and $q$ is fully faithful.  

    \item Assume that every for every morphism in $\tilde{\P}$, the exceptional pullback is fully faithful. Then for every morphism $f \in \op{Ind}^K_{\tilde{\P}}(\Ca)$ between the ind system $X_{\bb}:=f|_{[0]}$ and the constant system $f_{[1]}=Y_{\bb}=Y \in \Ca$ belongs to $\op{Adj}^*_*(\op{All})$. This essentially follows from the similar arguments that we did in the previous example. 
    \end{enumerate}
    
\end{example}
It turns out that, by functorial cohomological purity, the upper shriek and upper star extension agree for  objects in $\op{Pro}^K_{\S}(\Ca)$.
\begin{proposition}\label{uppershriekupperstarsmoothequivalence}
    Let $X_{\bb}: K \to \Ca \in \op{Pro}^K_{\S}(\Ca)$ and $\D$ be a presentable six-functor formalism and satisfying the conditions of \cref{Theorem:PurityFunctorialFullForm} such that $\Sigma_f$ is trivial for every morphism $f \in \S$. Then  we have an equivalence \begin{equation} \D^*(X_{\bb}) \cong \D^!(X_{\bb}) \end{equation} in the functor category $\op{Fun}(K^{\op{op}},\op{Pr}^L)$. Moreover, we have an 
     equivalence :
    \begin{equation}
      \op{colim}_{(-)^*}\D(X_{\bb}) \cong \op{colim}_{(-)^!} \D(X_{\bb})
    \end{equation}
    in $\op{Pr}^L$.
\end{proposition}
\begin{proof}
 We prove this in the following steps:
 \begin{enumerate}
     \item \textbf{Taking adjoints:}
     By considering adjoints, it is sufficient to prove
     \begin{equation}
        \D_{\#}(X_{\bb}) \cong \D_!(X_{\bb})
     \end{equation}
     in $\op{Fun}(K,\op{Cat}_{\infty})$.  
 \item \textbf{The cohomological Purity:} By \cref{Theorem:PurityFunctorialFullForm}, we have the composition:
 \begin{equation}
     \Delta^1 \times \K \xrightarrow{\op{id} \times X_{\bb}} \Delta^1 \times \Ca_{\S} \xrightarrow{\op{Pur}_{\D}} \op{Pr}^L_{\op{cl}}
 \end{equation}
 This is an equivalence which shows that $\D_{\#}(X_{\bb}) \cong \D_{!}(X_{\bb})$.
 \end{enumerate}
\end{proof}

\subsection{Extension Theorems.}

\begin{theorem}\label{theorem:IndProextensions}
 Let 
   \begin{equation}
       \D: \op{Corr}(\Ca)_{\E,\op{all}} \to \op{Pr}^L_{\op{cl}}
   \end{equation}
   be a presentable six-functor formalism  arising from a Nagata setup. Let $K$ be a filtered $\infty$-category.
   \begin{enumerate}
    
   \item  Let $\S \subset {\op{HL}}$ be another class of weakly stable maps in $\Ca$ satisfying Cohomological Purity. Then $\D$ can be extended to a presentable six-functor formalism
   \begin{equation}
       \D_{\op{pro}} : \op{Corr}(\op{Pro}^K_{\S}(\Ca))_{\op{Adj}^!_!(\E)),\op{All}} \to \op{Pr}^L_{\op{cl}}.
   \end{equation}
  \item  $\D$ can be extended to a presentable six-functor formalism
   \begin{equation}
       \D_{\op{ind}}: \op{Corr}(\op{Ind}^K_{\tilde{\P}}(\Ca))_{\E,\op{Adj}^*_*(\op{All})} \to \op{Pr}^L_{\op{cl}}.
   \end{equation}
   
\end{enumerate}
  
\end{theorem}
\begin{proof} 
    \begin{enumerate}
    \item We construct the extension in the following steps: 
    \begin{enumerate}
  
    \item As we have cohomological purity, we see that one can replace $\op{Adj}^!_!(\E) \cong \op{Adj}^*_!(\E)$. Replacing $\op{Corr}(-)$ by the simplicial set of squares (\cite[Proposition 4.2.3]{chowdhury2025sixfunctorformalismsiiiconstruction}) it follows to construct the presentable six-functor formalism 
        \begin{equation}
            \D_{\op{pro}} : \delta^*_{2,\{2\}}((\op{Pro}^K_{\S}(\Ca)^{\op{op}})^{\amalg,\op{op}})^{\op{cart}}_{\op{Adj}^*_!(\E),\op{All}} \to \op{Pr}^L_{\op{cl}}
        \end{equation} which extends $\D$. By construction of $(\Ca^{\op{op}})^{\amalg,\op{op}}$(\cite[Construction 2.4.3,1]{HA}), we see that:

        \begin{equation}
            \delta^*_{2,\{2\}}((\op{Pro}^K_{\S}(\Ca)^{\op{op}})^{\amalg,\op{op}})^{\op{cart}}_{\op{Adj}^*_!(\E),\op{All} }\cong \dd^*_{2,\{2\}}(\op{Fun}_{\S}(K,(\Ca^{\op{op}})^{\amalg,\op{op}}))^{\op{cart}}_{\op{Adj}^*_!(\E),\op{All}}
        \end{equation}
        
We construct the map  $\alpha_{\op{Pro}}$:
\begin{equation}
   \dd^*_{2,\{2\}}(\op{Fun}_{\S}(K,(\Ca^{\op{op}})^{\amalg,\op{op}}))^{\op{cart}}_{\op{Adj}^*_!(\E),\op{All}}\to \op{Fun}(K,\delta^*_{3,\{2,3\}}(\Ca^{\op{op}})^{\amalg,\op{op}})^{\op{cart}^*_!}_{\E,\S,\op{all}})
\end{equation}
where $\delta^*_{3,\{2,3\}}((\Ca^{\op{op}})^{\amalg,\op{cart}})^{\op{cart}^!_!}_{\E,\S\op{all}}$  defined as in \cref{corrsmoothcategory}. The map is constructed as follows:\\

 Firstly as $K \cong \op{colim}_{m \in \Delta_{/K}}\Delta^m$, it suffices to construct the map for $K=\Delta^m$. Let $\sigma_n$ be an $n$-simplex of L.H.S. This is a morphism of the form :
 $\Delta^n_{\E} \times (\Delta^n)^{\op{op}}_{\op{all}} \times \Delta^m_S \to (\Ca^{\op{op}})^{\amalg,\op{cart}}$ with specified directions of morphisms in $\S$ and $\E$. \\
 We define \[\alpha_{\op{Pro}}(\sigma_n): \Delta^n \times \Delta^m \to \delta^*_{3,\{2,3\}}((\Ca^{\op{op}})^{\amalg,\op{cart}})^{\op{cart}^*_!}_{\E,\S,\op{All}} \]
 as follows:
 Let $\beta : \Delta^l \to \Delta^n \times \Delta^m$. Then $\alpha_{\op{Pro}}(\beta)$ is given by composition of the following three chain of morphisms: 
 \begin{equation}
     \alpha_{\op{Pro}}(\beta):\Delta^l_{\E} \times \Delta^l_{\S} \times \Delta^l_{\op{All}} \to (\Delta^n \times \Delta^m)_{\E} \times (\Delta^n \times \Delta^m)^{\op{op}}_{\S} \times (\Delta^n \times \Delta^m)^{\op{op}}_{\op{All}}
 \end{equation}
 followed by 
 \begin{equation}
     (\Delta^n \times \Delta^m)_{\E} \times (\Delta^n \times \Delta^m)^{\op{op}}_{\S} \times (\Delta^n \times \Delta^m)^{\op{op}}_{\op{All}} \xrightarrow{\op{pr}} \Delta^n_{\E} \times (\Delta^m)^{\op{op}}_{\S} \times (\Delta^n)^{\op{op}}_{\op{All}} \cong \Delta^n_{\E} \times (\Delta^n)^{\op{op}}_{\op{all}} 
 \end{equation}
 and lastly by
 \begin{equation}
     \Delta^n_{\E} \times(\Delta^n)^{\op{op}}_{\op{all}} \times \Delta^m_S \xrightarrow{\sigma_n} (\Ca^{\op{op}})^{\amalg,\op{cart}}
 \end{equation}
 
 \item  As proven in \cref{Proposition: corrextendedformalismincludingsmooth}, we have map \[\D_{(\Ca,\S,\E)}: \dd^*_{3,\{2,3\}}\Ca^{\op{cart}^*_!}_{\E,\S,\op{all}} \to \op{Pr}^L \] where the maps on the directions of $\E,\S$ and $\op{All}$ go to $(-)_!,(-)^*,(-)^*$ respectively.Applying $\D'$ and taking colimits, we the following map
 \begin{equation}
      \alpha'_{\op{Pro}}:\op{Fun}(K,\delta^*_{3,\{2,3\}}((\Ca^{\op{op}})^{\amalg,\op{cart}})^{\op{cart}^*_!}_{\E,\S,\op{all}}) \xrightarrow{\D_{(\Ca,\S,\E)}} \op{Fun}(K,\op{Pr}^L) \xrightarrow{\op{colim}} \op{Pr}^L_{\op{cl}}.
     \end{equation}
     \item Composing $\alpha_{\op{Pro}}$ and $\alpha'_{\op{Pro}}$, we get our desired functor:
     \begin{equation}
         \D_{\op{pro}}:\op{Corr}(\op{Pro}^K_{\S}(\Ca))_{\op{Adj}^!_!(\E),\op{All}}  \xrightarrow{\alpha_{\op{Pro}}} \op{Fun}(K,\delta^*_{3,\{2,3\}}\Ca^{\op{cart}^*_!}_{\E,\S,\op{all}}) \xrightarrow{\alpha'_{\op{Pro}}} \op{Pr}_{\op{cl}}^L.
     \end{equation}
\end{enumerate}
\item The proof of this is dual to the proof of the previous theorem, define $\alpha_{\op{ind}}$ in the corresponding way, keeping track of edges in $\tilde{\P}$. We use \cref{Proposition: corrformalismextendingproper} and take $\op{colim}$ to get the functor to $\op{Pr}^L_{\op{cl}}$.
    \end{enumerate}
\end{proof}

\begin{remark}
\begin{enumerate}
\item The hypothesis that $K$ is filtered ensures that colimits are computed nicely. In particular, if the categories are compactly generated, then filtered colimits shall preserve compact generation. If one just wants to extend the functor without imposing closed properties or compact generation, one can do the above extension for an arbitrarily small simplicial set. 
    \item By construction, for an object $X_{\bb} \in \op{Pro}^K_{\S}(\Ca)$, we have 
    \begin{equation}
        \D_{\op{pro}}(X_{\bb}) \cong \op{colim}_{\op{Pr}^L_{\op{cl}}}^* \D(X_k).
    \end{equation} This agrees with the definition in \cite{yaylali2025rationalmotivesproalgebraicstacks}.
    \item In the similar fashion, for an object $Y_{\bb} \in \op{Ind}^K_{\tilde{\P}}(\Ca)$, we have 
    \begin{equation}
        \D_{\op{ind}}(Y_{\bb}) \cong \op{lim}^!_{\op{Pr}^L_{\op{cl}}}\D(Y_k).
    \end{equation}
    This agrees with upper shriek extension as done for prestacks in \cite[Section 2.2]{Richarz_2021}.
\end{enumerate}
    
\end{remark}
In our applications, one usually wants to extend six-functor formalism, which arises as extensions of Nagata setups. Using the notion of nice geometric and exceptional pairs (\cref{Definition:Nice geometric and exceptional pairs.}), we have the following theorem, which is a consequence of the one above.
\begin{theorem}\label{theorem:proextensiongeoandeceptionalpairs}
Let 
   \begin{equation}
       \D: \op{Corr}(\Ca)_{\E,\op{all}} \to \op{Pr}^L_{\op{cl}}
   \end{equation}
   be a presentable six-functor formalism  arising from a Nagata setup. Let $K$ be a filtered $\infty$-category which admits an initial object $0$.
   \begin{enumerate}
   \item Let $(\Ca,\S,\E) \subset (\Ca,\S',\E')$ be a nice geometric pair. Assume that $\D$ satisfies the \textit{continuity property} i.e. for $X \cong \op{lim}_i X_i$ where transition maps are in $\S$, we have 
   \begin{equation}
       \D(X) \cong \op{colim}^* \D(X_i),
   \end{equation}then $\D_{\op{pro}}$ extends to 
   \begin{equation}
       \D_{\op{pro}} : \op{Corr}(\op{Pro}^K_{\S'}(\Ca'))_{\op{Adj}^!_!(\E'),\op{all}} \to \op{Pr}^L_{\op{cl}}
   \end{equation}
   which agrees with the extension $\D$ when restricted to $\Ca'$.
   \item Let $(\Ca,\S,\E) \subset (\Ca,\S,\E')$ be an exceptional pair. Assume that $\D$ satisfies the continuity property as before, then $\D_{\op{pro}}$ extends to

    \begin{equation}
       \D_{\op{pro}} : \op{Corr}(\op{Pro}^K_{\S}(\Ca))_{\op{Adj}^!_!(\E'),\op{all}} \to \op{Pr}^L_{\op{cl}}.
   \end{equation}

   \end{enumerate}

\end{theorem}

\begin{proof}
    Firstly, one checks that in each of these pairs, the corresponding pairs on the Pro-Categories are also nice geometric and exceptional pairs. Then it follows from \cref{Theorem:extensionalonggeometricandexceptionalpairs}.
\end{proof}

In the case of Ind setups, one can extend formalisms without the Nagata setup condition. The theorem is as follows.

\begin{theorem}\label{Theorem:ExtendingIndGeneralSetup}
    Let 
    \begin{equation}
        \D : \op{Corr}(\Ca)_{\E,\op{all}} \to \op{Pr}^L_{\op{cl}}
    \end{equation}
    be a presentable six-functor formalism. Let $\tilde{P}$ be a collection of morphisms in $\E$ such that $f_! \cong f_*$. Then $\D$ can be extended to a  six-functor formalism
    \begin{equation}
        \D_{\op{ind}} : \op{Corr}(\op{Ind}^K_{\tilde{P}}(\Ca))_{\E,\op{Adj}^*_*(\op{All})} \to \op{Pr}^L_{\op{cl}}
    \end{equation}
\end{theorem}
\begin{proof}
    As seen in the proof of \cref{theorem:IndProextensions}, it is sufficient to construct 
    \begin{equation}
        \D_{(\Ca,\tilde{P},\E)} : \dd^*_{3,\{3\}}((\Ca^{\op{op}})^{\amalg,\op{op}})^{\op{cart}^*_*}_{\E,\tilde{P},\op{All}} \to \op{Pr}^L_{\op{cl}}    \end{equation}

        Firstly, we can see we have the canonical functor

        \begin{equation}
            \dd^*_{3,\{2,3\}}((\Ca^{\op{op}})^{\amalg,\op{op}})^{\bx}_{\E,\tilde{P},\op{all}} \to \op{Pr}^L_{\op{cl}} 
        \end{equation}

where $\bx$ is the tiling given by cartesian squares along $(\E,\tilde{P})$ and $(\E,\op{all})$ directions and the squares satisfying $(-)_*,(-)^*$ compatibility in $(\P,\op{all})$ direction. The morphism sends $(-)_!$ in direction of $\E$ and $(-)^*$ on $\tilde{P},\op{all}$ directions. Using the theorem of partial adjoints, we can construct a functor which switches the morphisms from $(-)^*$ to $(-)_*$ in the direction of $\tilde{P}$. Using the compactification theorem, we can extend the functor to
\begin{equation}
    \D_{(\Ca,\tilde{P},\E}:  \dd^*_{3,\{3\}}((\Ca^{\op{op}})^{\amalg,\op{op}})^{\op{cart}^*_*}_{\E,\tilde{P},\op{all}} \to \op{Pr}^L_{\op{cl}} 
\end{equation}
which sends now $(-)_!$ on $\E,\tilde{P}$ direction and $(-)^*$ on direction $\op{All}$. This completes the proof. 

\end{proof}

\section{Application: Motivic Homotopy Theory  of Ind-pro-Algebraic Stacks.}

In this section, we apply our extension theorems from schemes to  ind-pro algebraic stacks to the motivic homotopy theory ($\SH (-)$). 
\subsection{Motivic Homotopy Theory for Algebraic Stacks.}
To a scheme $S$, Morel and Voevodsky (\cite{Morel-Voevodsky})introduced the motivic stable homotopy category $\op{SH}(S)$. The groundbreaking work of Voevodsky and Ayoub (\cite{Ayoub_6FF_vol1},\cite{Ayoub_6FF_vol2}) shows that the assignment $ S \mapsto \op{SH}(S)$ admits a six-functor formalism. Robalo, in \cite{Robalothesis} formulated the above results in the language of $(\infty,1)$-categories which gives us a abstract 6-functor formalism :
\begin{equation}
    \SH (-) : \op{Corr}(\Schf)^{\otimes}_{\op{sft},\op{all}} \to \op{Cat}^{\otimes}_{\infty}.
\end{equation}

In recent years, several extensions of stable motivic homotopy theory have been to algebraic stacks. Two kinds of $\SH$ have been studied in the context of algebraic stacks, namely the \textit{genuine} $\SH$ (\cite{Hoyois_Equiv_Six_Op}) and the \textit{Borel} $\SH$. In this article, we only consider the Borel $\SH$ constructed by Chowdhury(\cite{Chowdhury}) and Khan-Ravi(\cite{Khan-Ravi_Generalised_Coh_Stacks}). 

Unlike $ \D_{et}$, $ \SH$ satisfies Nisnevich descent, not étale descent. We define a category of algebraic stacks that encodes the Nisnevich condition in the atlas.

	\begin{definition}\cite[Definition 2.1.4]{Chowdhury}
		\begin{enumerate}
			\item [$ (i) $]	Let $ \op{AlgSt}^{NL} $ be the $\infty$-category of algebraic (1)-stacks (also known as Artin stacks) for which there exists an atlas that admits Nisnevich-local sections. In other words, the category consists of algebraic stacks $\X$ for which there exists an atlas $ x: X \to \X$ such that for any test scheme $T$ with a morphism $ t: T \to \X$, the base change morphism $x': X \times_{\X} T \to T$ admits sections Nisnevich-locally.     
			\item [$ (ii) $]			 A morphism of algebraic stacks $f : \X \to \Y$ in $\op{ASt}^{ NL}$ is said to admit \textit{Nisnevich-local sections} if there exists an atlas $ y : Y \to \Y$ admitting Nisnevich-local sections and a morphism $s : Y \to \X$ such $ f\circ s = y$. We will denote the set of these maps as $ \op{NL} $.
			\item [$(iii)$] The topology generated by NL-maps will be denoted as $\tau_{NL}$ (or simply NL if it is clear from the context that we are talking about the topology)and it will be referred to as the \textit{Nisnevich Local} topology. Covers for $\tau_{NL}$ will be called NL-covers.
		\end{enumerate}
        \end{definition}
 This category includes interesting examples of algebraic stacks, such as quotient stacks.  However, it turns out that we have the following general statement:
 \begin{theorem}\cite[Theorem A.1]{deshmukh2025motivichomotopytypealgebraic},\cite[Proposition 2.4] {chowdhury2025nonrepresentablesixfunctorformalisms}

 Let $\X$ be an algebraic stack. Then $\X  \in \op{AlgSt}^{NL}$. Hence $\op{AlgSt}= \op{AlgSt}^{NL}$.
     
 \end{theorem}

The above theorem shows that every algebraic stack admits an atlas that admits sections Nisnevich locally. The main result of the extension of the six-functor formalism in the motivic context can now be stated as follows:
\begin{theorem}[\cite{Chowdhury,Khan-Ravi_Generalised_Coh_Stacks}]\label{SHonalgstacks}
$\SH_{(\Schf,\op{sft})}(-)$ extends to an abstract $6$-functor formalism 
\begin{equation}
    \SH(-) : \op{Corr}(\op{AlgSt})^{\otimes}_{\op{lft},\op{All}} \to \op{Cat}^{\otimes}_{\infty}
\end{equation}
Where $\op{lft}$ is the set of maps of algebraic stacks which are not necessarily representable. In particular, for any algebraic stack admitting a Nis-loc atlas, $\SH(\X)$ can be described as follows:
\begin{equation}
    \SH(\X) \cong \op{lim}_{\bb \in \Delta} \SH(X^{\bb}_{\X}).
    \end{equation}

\end{theorem}

\subsection{$\SH(-)$ for Ind-Pro-Algebraic Stacks.}

We apply a combination of the extension theorems from our last section. Before we proceed to the extension theorems, let us recall the key examples relevant to this setup. We recall the objects as introduced in \cite[Section 4]{richarzscholbachintersectionmotive}.
Let $G$ be an affine $\mathbf{Z}$-group scheme whose geometric fibers are connected reductive groups and which admits a maximal torus $T$ defined over $\mathbf{Z}$.
\begin{enumerate}
    \item The loop group $LG$ is a group functor on the category of rings defined as :
    \begin{equation}
        LG : R \mapsto G(R((\varpi)))
    \end{equation}
    where $R((\varpi))$ is the ring of Laurent series in variable $\varpi$.
    As $G$ is an affine group scheme and of finite type, $LG$ is represented by an ind-affine ind-scheme.
    \item The positive loop group is a group functor on the category of rings defined as :
    \begin{equation}
        L^+G : R \mapsto G(R[|\varpi|])
    \end{equation}
    where $R[|\varpi|] \subset R((\varpi))$ is the subring of formal power series. As presheaves, we have 
    \begin{equation}
    L^+G = \varinjlim_i G_i
    \end{equation}
    where $G_i(R) = G(R[|\varpi|])/(\varpi^{i+1})$. This $L^+G$ is represented by a pro-algebraic group $\mathbf{Z}$-group scheme. 
    \item The inclusion $L^+G \subset LG$ is represented by a closed immersion. 
    \item The \textit{affine Grassmanian} $\op{Gr}_G$ is defined as 
    \begin{equation}
        \op{Gr}_G := \frac{LG}{L^+G}.
    \end{equation} 
    The prestack $\op{Gr}_G$ is an ind-projective $\mathbf{Z}$-scheme(\cite[Section 4.3]{richarzscholbachintersectionmotive}). There is an $L^+G$-action on $\op{Gr}_G$ given by 
    \begin{equation}
        L^+G \times \op{Gr}_G \to \op{Gr}_G \quad (g,x) \mapsto g.x
    \end{equation}
    By \cite[Lemma 4.3.3]{richarzscholbachintersectionmotive}, it follows that the action of $L^+G$ on $\op{Gr}_G$ gives a $L^+G$-stable presentation of 
    \begin{equation}
        \op{Gr}_G \cong \op{colim}_i \op{Gr}_{G,i}
    \end{equation}
    such that the $L^+G$ action on $\op{Gr}_{G,j}$ factors via $G_i$ for some $i$. Morever each $\op{Gr}_{G,j}$ is a projective $\mathbf{Z}$-scheme.
    \item The Hecke stack $\op{Hk}_G:=\op{Gr}_G/L^+G = L^+G \backslash LG /L^+G$ is a ind-pro algebraic stack :
    \begin{equation}\label{Heckestackindprodescription}
        \frac{\op{Gr}_G}{L^+G} \cong "\varprojlim_i" \op{Gr}_{G,i}/L^+G \cong "\varprojlim_i" "\varinjlim_j" \Big[\frac{\op{Gr}_{G,i}}{G_j}\Big]
    \end{equation}

\item  The Hecke stack $\op{Hk}_G$ also can be interpreted as fiber product $BL^+G \times^{BLG} BL^+G$.
\end{enumerate}
Let $K$ be a filtered $\infty$-category with an initial object. Let $(\Ca,\E,\I,\P)= (\op{Sch},\op{sft},\op{open},\op{prop})$ and  $\Ca' = \op{AlgSt}$. Then, using  $\S (\S')$ to be (representable) smooth morphisms $\E$ to be (rep) for locally finite type maps of schemes (algebraic stacks) in the way as done \cite{chowdhury2025nonrepresentablesixfunctorformalisms}, along with \cref{theorem:proextensiongeoandeceptionalpairs}, we get
\begin{notation}
  
 Let    $\op{Pro}^K_{\S}(\op{AlgSt})$ be the category of $K$-pro-algebraic stacks with smooth surjections as transition maps.
    \end{notation}
We can then extend \cref{SHonalgstacks} via \cref{theorem:proextensiongeoandeceptionalpairs} to the following proposition:
\begin{proposition}
    $\SH(-)$ extends to a presentable six-functor formalism
    \begin{equation}
        \SH(-): \op{Corr}(\op{Pro}^K_{\S}(\op{Algst}))_{\op{Adj}^!_!(\E),\op{All}} \to \op{Pr}^L_{\op{cl}}
    \end{equation}
    In particular for any pro-algebraic stack $\X:=\{\X_k\}_{k \in K}$, we have
    \begin{equation}
        \SH(\X) \cong \op{colim}_{\op{Pr}^L_{\op{cl}}}^*\SH(\X_i) .
    \end{equation}
\end{proposition}
We want to extend the formalism to the setting of Ind-Pro algebraic stacks. Let $K$ be a filtered $\infty$-category. Set $\tilde{\P} = \op{Ind}(\op{Adj}^!_!(\op{Pro}(\op{cl})))$ where $\op{cl} \subset \P$ the category of closed immersions. 

\begin{notation}
    Set 
    \begin{enumerate}
        \item   $\op{IndPro}^K_{\op{cl},\op{sm}}(\op{AlgSt}):= \op{Ind}^K_{\op{cl}}(\op{Pro}^K_{\S}(\op{AlgSt}))$
        \item  $\op{IndPro}^!_!(\E):= \op{Adj}^!_!(\E))$ and
        \item $\op{IndPro}^*_!(\op{All}) := \op{Adj}^*_!(\op{All})$.
    \end{enumerate}
   
\end{notation}

Via \cref{Theorem:ExtendingIndGeneralSetup}, we get
\begin{theorem}\label{theorem:sixfuncindproSH}
    $\SH(-)$ extends to a presentable six-functor formalism
    \begin{equation}
        \SH(-) : \op{Corr}(\op{IndPro}^K_{\op{cl},\op{sm}}(\op{AlgSt}))_{\op{IndPro}^!_!(\E),\op{IndPro}^*_!(\op{All})} \to \op{Pr}^L_{\op{cl}}
    \end{equation}
    In particular for an Ind-Pro-Algebraic stack $\X:=\{X_{ij}\}_{i \in K,j\in K}$ we have

\begin{equation}
    \SH(\X) \cong \op{colim}_{!,i \in K}\op{colim}^*_{j \in K}\SH(\X_{ij}) \in \op{Pr}^L_{\op{cl}}
\end{equation}
\end{theorem}

\begin{definition}
The \textit{Spectral Satake Category} is defined as $\SH(\op{Hk}_G)$. In particular, using the Ind-pro-description of the Hecke stack (\cref{Heckestackindprodescription}), we have

\begin{equation}
    \SH(\op{Hk}_G) \cong \op{colim}_{!,i} \op{colim}^*_{j} \SH([\frac{Gr_{G,i}}{G_j}])
\end{equation}

\end{definition}

\begin{remark}
    \begin{enumerate}
    \item As we have taken filtered colimits, $\SH(\X)$ is also stable. 
    \item As $\op{SH}(-)$ satisfies continuity, one sees that if an algebraic stack $\X$ admits a projective system, then taking the colimit of the computation agrees with $\SH(\X)$. 
        \item One can do the same construction for $\op{DM}(-)$ also in the similar fashion. This can be done as $\op{DM}= \op{Mod}_{H\mathbf{Z}}(\SH)$. This generalizes the work of Yaylali (\cite{yaylali2025rationalmotivesproalgebraicstacks}) to integral coefficients. 
        \item  As $\SH(-)$ can be defined for derived stacks (\cite{chowdhury2025nonrepresentablesixfunctorformalisms}[Section 5.1], one can do the same construction for inductive systems and projective systems of derived algebraic stacks.
        \item As the pullback map for smooth surjections and the pushforward map for closed immersions is fully faithful, by \cref{example:proandindadjointablemorphisms} it follows that we have exceptional pushforward for the non-representable morphism $\op{BL^+G} \to \op{pt}$. 
    \end{enumerate}
\end{remark}

\bibliographystyle{abbrv}

\end{document}